\documentclass[10pt]{amsart}

\usepackage[latin1]{inputenc}
\usepackage[T1]{fontenc}

\usepackage[theoremfont]{baskervillef}
\usepackage[baskerville, varg]{newtxmath}
\usepackage[scr=boondox,cal=boondoxo,bb=boondox,frak=euler]{mathalfa}

\usepackage{hyperref}
\usepackage{amsmath,amssymb,amsthm,amsfonts,dsfont,mathrsfs}
\usepackage{cleveref}

\usepackage[style=alphabetic,natbib=true,backend=biber]{biblatex}

\usepackage{graphicx}
\usepackage{multirow}

\allowdisplaybreaks[4]

\makeatletter

\newsavebox\myboxA
\newsavebox\myboxB
\newlength\mylenA

\newcommand*\xoverline[2][0.75]{%
    \sbox{\myboxA}{$\m@th#2$}%
    \setbox\myboxB\null
    \ht\myboxB=\ht\myboxA%
    \dp\myboxB=\dp\myboxA%
    \wd\myboxB=#1\wd\myboxA
    \sbox\myboxB{$\m@th\overline{\copy\myboxB}$}
    \setlength\mylenA{\the\wd\myboxA}
    \addtolength\mylenA{-\the\wd\myboxB}%
    \ifdim\wd\myboxB<\wd\myboxA%
       \rlap{\hskip 0.5\mylenA\usebox\myboxB}{\usebox\myboxA}%
    \else
        \hskip -0.5\mylenA\rlap{\usebox\myboxA}{\hskip 0.5\mylenA\usebox\myboxB}%
    \fi}

\makeatother

\newcommand{\A}{\mathcal{A}}
\newcommand{\Z}{\mathbb{Z}}
\newcommand{\odd}{{\mathrm{odd}}}
\newcommand{\even}{{\mathrm{even}}}
\renewcommand{\H}{\mathcal{H}}
\newcommand{\MRV}{\mathrm{MRV}}
\newcommand{\RV}{\mathrm{RV}}
\newcommand{\bRV}{\overline{\mathrm{RV}}}
\newcommand{\Mod}{\mathrm{Mod}}
\newcommand{\Sp}{\mathrm{Sp}}
\newcommand{\SL}{\mathrm{SL}}
\newcommand{\R}{\mathbb{R}}
\newcommand{\bOmega}{\xoverline{\Omega}}
\newcommand{\be}{\xoverline{e}}
\newcommand{\bV}{\xoverline{V}}
\newcommand{\bT}{\xoverline[0.9]{T}}
\renewcommand{\mod}{\mathbin{\mathrm{mod}}}
\newcommand{\NS}{\mathrm{NS}}
\newcommand{\Id}{\mathrm{Id}}
\newcommand{\tr}{{\mbox{\raisebox{0.33ex}{\scalebox{0.6}{$\intercal$}}}}}
\let\Sigmaaux\Sigma
\renewcommand{\Sigma}{\mbox{\scalebox{1.05}{$\Sigmaaux$}}}
\newcommand{\RR}{\mathcal{R}}
\newcommand{\C}{\mathbb{C}}

\newtheorem{thm}{Theorem}[section]
\newtheorem{lem}[thm]{Lemma}
\newtheorem{prop}[thm]{Proposition}
\newtheorem{cor}[thm]{Corollary}
\newtheorem{fact}{Fact}
\theoremstyle{definition}
\newtheorem{defn}[thm]{Definition}
\theoremstyle{remark}
\newtheorem{rem}[thm]{Remark}

\DeclareFieldFormat{title}{\mkbibquote{#1}}

\begin{document}

	\title{Classification of Rauzy--Veech groups: proof of the Zorich conjecture}
	
	\author{Rodolfo Gutiérrez-Romo}
	\address{Institut de Mathématiques de Jussieu -- Paris Rive Gauche, UMR 7586, Bâtiment Sophie Germain, 75205 PARIS Cedex 13, France.} \email{rodolfo.gutierrez@imj-prg.fr}
	
	\begin{abstract}
		We classify the Rauzy--Veech groups of all connected components of all strata of the moduli space of translation surfaces in absolute homology, showing, in particular, that they are commensurable to arithmetic lattices of symplectic groups. As a corollary, we prove a conjecture of Zorich about the Zariski-density of such groups.
	\end{abstract}
	
	\maketitle
	\markboth{R. GUTIÉRREZ-ROMO}{CLASSIFICATION OF RAUZY--VEECH GROUPS}
	
	\section{Introduction}
	
	The Kontsevich--Zorich conjecture states the simplicity of the Lyapunov spectra of almost all interval exchange transformations or translation flows with respect to the Masur--Veech measures. After an important partial progress by Forni in \citeyear{F:deviations} \cite{F:deviations}, it was established in full generality by Avila and Viana in \citeyear{AV:KZ_conjecture} \cite{AV:KZ_conjecture}. Nevertheless, their methods leave open a stronger conjecture of Zorich from \citeyear{Z:wind}: the Zariski-density of the Rauzy--Veech groups \cite{Z:wind}. Indeed, the main ingredient of Avila and Viana's proof is the fact that Rauzy--Veech groups are pinching and twisting, which is implied by Zariski-density by the work of Benoist \cite{B:density_pinching}. However, the converse is not true: there are known examples of pinching and twisting groups with small Zariski closure \cite[Appendix A]{AMY:hyperelliptic}.
	
	The pioneering work of Avila, Matheus and Yoccoz \cite{AMY:hyperelliptic} shows that this conjecture holds for the particular case of hyperelliptic Rauzy--Veech groups. Their methods also laid the groundwork for further results in this direction.
	
	In this article we answer this conjecture affirmatively by establishing a stronger statement. Indeed, our main result is the following theorem:
	\smallbreak
	\begin{thm} \label{thm:general}
		At the level of absolute homology, the Rauzy--Veech group of any connected component of any stratum of the moduli space of genus $g$ translation surfaces is an explicit finite-index subgroup of $\Sp(2g, \Z)$. In particular, if $g \geq 3$ and the stratum is connected, it is equal to its entire ambient symplectic group.
	\end{thm}
	This theorem its proved in various steps along the article. A summary of our results is presented in \Cref{thm:classification}.
	\smallbreak
	
	We obtain a similar result for the monodromy group, since it contains the Rauzy--Veech group. The Zariski-density was known in this case \cite[Corollary 1.7]{F:zero_exponents}.
	\smallbreak
	\begin{cor}
		The monodromy group of any connected component of any stratum of the moduli space of genus $g$ translation surfaces is commensurable to $\Sp(2g, \Z)$. In particular, if $g \geq 3$ and the stratum is connected, it is equal to its entire ambient symplectic group.
	\end{cor}
	\smallbreak
	
	To prove \Cref{thm:general}, we reduce the general conjecture to the case of minimal strata. Indeed, we show in \Cref{prop:reduction} that the Rauzy--Veech group of any connected component of any stratum contains the Rauzy--Veech group of specific connected components of minimal strata of surfaces of the same genus. For the case of minimal strata, we prove the following theorem:
	\smallbreak
	\begin{thm} \label{thm:minimal}
		The Rauzy--Veech group of a non-hyperelliptic connected component of a minimal stratum of the moduli space of genus $g$ translation surfaces is the preimage of the orthogonal group $O(Q) \subseteq \Sp(2g, \Z / 2\Z)$ by the modulus-two reduction $\Sp(2g, \Z) \to \Sp(2g, \Z / 2\Z)$, where $Q$ is the usual quadratic form on modulus-two homology.
	\end{thm}
	\smallbreak
	An immediate corollary is that these groups have finite index in $\Sp(2g, \Z)$. The proof of \Cref{thm:minimal} is elementary and is divided into two parts. First, we show that the level-two congruence subgroup of $\Sp(2g, \Z)$ (that is, the kernel of the modulus-two reduction) is contained in such Rauzy--Veech groups in \Cref{lem:kernel1} and \Cref{lem:kernel2}. Then, we show that the modulus-two reduction is surjective onto $O(Q)$ in \Cref{lem:surjectivity}. Both parts of the proof rely on constructing explicit sets of generators inductively, first for the level-two congruence subgroup and then for $O(Q)$.
	
	Our methods also allow us to describe the Rauzy--Veech groups of the non-hyperelliptic connected components of $\H(g-1, g-1)$ for $g \geq 3$ at the level of relative homology explicitly. Indeed, we show in \Cref{thm:odd_d} that they can be written as suitable subgroups of semi-direct products.
	
	Besides giving a new proof of Avila and Viana's theorem establishing the Kontsevich--Zorich conjecture, our main result has other applications: for instance, Magee \cite{M:gap} proved that the Zorich conjecture implies the uniform spectral gap property (and the uniform rate of mixing of the Teichmüller geodesic flow) of congruence covers of connected components of the strata of the moduli space of translation surfaces.
	\smallbreak
	
	The article is organized as follows. In \Cref{sec:preliminaries} we review the basic definitions and background. Sections \ref{sec:leveltwo} and \ref{sec:orthogonal} contain the two parts of the proof of \Cref{thm:minimal}: respectively, that the Rauzy--Veech group of a minimal stratum contains the level-two congruence subgroup and that its modulus-two reduction is surjective onto the corresponding orthogonal group. In \Cref{sec:odd_d} we give an explicit description of the Rauzy--Veech groups of the non-hyperelliptic components of $\H(g - 1, g - 1)$ for $g \geq 3$ at the level of relative homology. In \Cref{sec:reduction} we complete the proof of the main theorem by showing that the general conjecture can be reduced to the case of minimal strata. Finally, we make some computations explicit for the base cases of the induction used to prove \Cref{thm:minimal} in the  \nameref{sec:appendix}.
	
	\section{Rauzy--Veech groups} \label{sec:preliminaries}
	
	\subsection{Rauzy induction} In this section we will briefly recall the Rauzy induction algorithm on permutations, the definition of Kontsevich--Zorich matrices and Rauzy--Veech groups. We refer the reader to the lecture notes by Yoccoz \cite{Y:pisa} and the survey by Viana \cite{V:iet} for more details.
	
	Let $\A$ be a finite set of cardinality $d \geq 3$. Let $\pi = (\pi_\mathrm{t}, \pi_\mathrm{b})$ be a pair of bijections $\A \to \{1, \dotsc, d\}$, which can be interpreted as a permutation. We write $\pi = (\pi_\mathrm{t}, \pi_\mathrm{b})$ as
	\[
		\pi =
		\begin{pmatrix}
			\alpha_{\mathrm{t}, 1} & \alpha_{\mathrm{t}, 2} & \cdots & \alpha_{\mathrm{t}, d} \\
			\alpha_{\mathrm{b}, 1} & \alpha_{\mathrm{b}, 2} & \cdots & \alpha_{\mathrm{b}, d}
		\end{pmatrix},
	\]
	where $\alpha_{\varepsilon,j} = \pi_\varepsilon^{-1}(j)$ for $\varepsilon \in \{\mathrm{t}, \mathrm{b}\}$ and $j \in \{1, \dotsc, d\}$. We define an alternate form $\Omega_\pi$ indexed by $\A \times \A$ as:
	\[
		(\Omega_\pi)_{\alpha\beta} =
		\begin{cases}
			+1 & \pi_{\mathrm{t}}(\alpha) < \pi_{\mathrm{t}}(\beta) \text{ and } \pi_{\mathrm{b}}(\alpha) > \pi_{\mathrm{b}}(\beta) \\
			-1 & \pi_{\mathrm{t}}(\alpha) > \pi_{\mathrm{t}}(\beta) \text{ and } \pi_{\mathrm{b}}(\alpha) < \pi_{\mathrm{b}}(\beta) \\
			0 & \text{otherwise.}
		\end{cases}
	\]
	
	We say that $\pi = (\pi_\mathrm{t}, \pi_\mathrm{b})$ is \emph{irreducible} if $\pi_\mathrm{t}^{-1}(\{1, \dotsc, j\}) \neq \pi_\mathrm{b}^{-1}(\{1, \dotsc, j\})$ for any $1 \leq j < d$. Moreover, we say that $\pi$ is \emph{degenerate} if there exists $1 \leq j < d$ such that one of the following three conditions holds:
	\[
		\pi_{\mathrm{b}}(\alpha_{\mathrm{t},j}) = d, \quad \pi_{\mathrm{b}}(\alpha_{\mathrm{t},j+1}) = 1 \quad \text{ and } \quad \pi_{\mathrm{b}}(\alpha_{\mathrm{t},1}) = \pi_{\mathrm{b}}(\alpha_{\mathrm{t},d}) + 1;
	\]
	\[
		\pi_{\mathrm{b}} (\alpha_{\mathrm{t},j+1}) = 1 \quad \text{ and } \quad \pi_{\mathrm{b}}(\alpha_{\mathrm{t},1}) = \pi_{\mathrm{b}}(\alpha_{\mathrm{t},j}) + 1;
	\]
	or
	\[
		\pi_{\mathrm{b}}(\alpha_{\mathrm{t},j}) = d \quad \text{ and } \quad \pi_{\mathrm{b}}(\alpha_{\mathrm{t},j+1}) = \pi_{\mathrm{b}}(\alpha_{\mathrm{t},d}) + 1.
	\]
	Otherwise, $\pi$ is said to be \emph{nondegenerate}. From now on, we will always assume that a pair of bijections is both irreducible and nondegenerate.
	
	The Rauzy induction algorithm consists of two operations on $\pi$, which we call \emph{top} and \emph{bottom}. The top operation is defined as:
	\[
		R_{\mathrm{t}}(\pi) =
		\begin{pmatrix}
			\alpha_{\mathrm{t}, 1} & \cdots & \alpha_{\mathrm{t}, k-1} & \alpha_{\mathrm{t}, k} & \alpha_{\mathrm{t}, k+1} & \alpha_{\mathrm{t},k+2} & \cdots & \alpha_{\mathrm{t}, d} \\
			\alpha_{\mathrm{b}, 1} & \cdots & \alpha_{\mathrm{b}, k-1} & \alpha_{\mathrm{t},d} & \alpha_{\mathrm{b},d} & \alpha_{\mathrm{b},k+1} & \cdots & \alpha_{\mathrm{b}, d-1} \\
		\end{pmatrix},
	\]
	where $k$ satisfies $\alpha_{\mathrm{b}, k} = \alpha_{\mathrm{t}, d}$. In this case, we call $\alpha_{\mathrm{t}, d}$ the \emph{winner} and $\alpha_{\mathrm{b}, d}$ the \emph{loser} of the algorithm.
	
	The bottom operation is defined as:
	\[
		R_{\mathrm{b}}(\pi) =
		\begin{pmatrix}
			\alpha_{\mathrm{t}, 1} & \cdots & \alpha_{\mathrm{t}, k-1} & \alpha_{\mathrm{b}, d} & \alpha_{\mathrm{t}, d} & \alpha_{\mathrm{t},k+1} & \cdots & \alpha_{\mathrm{t}, d-1} \\
			\alpha_{\mathrm{b}, 1} & \cdots & \alpha_{\mathrm{b}, k-1} & \alpha_{\mathrm{b},k} & \alpha_{\mathrm{b},k+1} & \alpha_{\mathrm{b},k+2} & \cdots & \alpha_{\mathrm{b}, d} \\
		\end{pmatrix},
	\]
	where $k$ satisfies $\alpha_{\mathrm{t}, k} = \alpha_{\mathrm{b}, d}$. In this case, we call $\alpha_{\mathrm{b}, d}$ the \emph{winner} and $\alpha_{\mathrm{t}, d}$ the \emph{loser} of the algorithm.
	
	Observe that if $\pi$ is irreducible and nondegenerate, then $R_{\mathrm{t}}(\pi)$ and $R_{\mathrm{b}}(\pi)$ are irreducible and nondegenerate.
	
	Consider now the directed graph whose vertices are the irreducible, nondegenerate pairs and such that $\pi \to \pi'$ is an edge if $R_{\mathrm{t}}(\pi) = \pi'$ or $R_{\mathrm{b}}(\pi) = \pi'$. The connected (or, equivalently, strongly connected) components of this graph are called \emph{Rauzy classes}. They are in a finite-to-one correspondence with the connected components of the strata of the moduli space of translation surfaces as was proven by Veech \cite{V:gauss}.
	
	\subsection{Rauzy--Veech groups in homology}
	
	Let $\RR$ be a Rauzy class. We consider an undirected version $\tilde{\RR}$ of $\RR$: for each arrow $\gamma = \pi \to \pi'$ we add a reversed arrow $\gamma^{-1} = \pi' \to \pi$. Now, let $\gamma = \pi \to \pi'$ be an arrow in $\RR$. Let $\alpha_{\mathrm{w}}$ and $\alpha_{\mathrm{l}}$ be, respectively, the winner and loser of the operation sending $\pi$ to $\pi'$. We define the \emph{Kontsevich--Zorich matrix} indexed by $\A \times \A$ as $B_{\gamma} = \Id + E_{\alpha_{\mathrm{l}}\alpha_{\mathrm{w}}} \in \SL(\Z^\A)$, where $\Id$ is the identity matrix and $E_{\alpha_{\mathrm{l}}\alpha_{\mathrm{w}}}$ has only one non-zero coefficient, equal to $1$, at position $\alpha_{\mathrm{l}}\alpha_{\mathrm{w}}$. Analogously, we define $B_{\gamma^{-1}} = B_{\gamma}^{-1} = \Id - E_{\alpha_{\mathrm{l}}\alpha_{\mathrm{w}}} \in \SL(\Z^\A)$. Now consider a walk $\gamma = \gamma_1 \gamma_2 \dotsb \gamma_n$ in $\tilde{\RR}$ starting at $\pi$ and ending at $\pi'$. We define $B_\gamma = B_{\gamma_n} B_{\gamma_{n-1}} \dotsb B_{\gamma_1} \in \SL(\Z^\A)$, which satisfies $\Omega_{\pi'} = B_\gamma \Omega_\pi B_\gamma^\tr$. In particular, if $\pi' = \pi$ (that is, if $\gamma$ is a cycle), one has that $B_\gamma$ (acting on \emph{row} vectors) belongs to $\Sp(\Omega_\pi, \Z)$. The Rauzy--Veech group of $\pi$ is the group generated by matrices of this form:
	\begin{defn}
		Let $\RR$ be a Rauzy class and $\pi \in \RR$ be a fixed vertex. We define the \emph{Rauzy--Veech group} $\RV(\pi)$ of $\pi$ as the set of matrices of the form $B_\gamma \in \Sp(\Omega_\pi, \Z)$ where $\gamma$ is a cycle on $\tilde{\RR}$ with endpoints at $\pi$. We will always consider the action of $\RV(\pi)$ on row vectors unless explicitly stated.
	\end{defn}
	
	Observe if $\pi, \pi'$ are vertices of the same Rauzy class $\RR$, then $\RV(\pi)$ and $\RV(\pi')$ are isomorphic, so we can define the Rauzy--Veech group of a Rauzy class. Indeed, if $\gamma$ is any walk joining $\pi$ and $\pi'$, then the conjugation by $B_\gamma$ is an isomorphism between $\Sp(\Omega_\pi, \Z)$ and $\Sp(\Omega_{\pi'}, \Z)$ and between $\RV(\pi)$ and $\RV(\pi')$. This shows, in particular, that the Rauzy--Veech group of a Rauzy class has a well-defined index inside its ambient symplectic group.
	
	\begin{rem}
		One could also define Rauzy--Veech \emph{monoids} by considering loops on the directed graph $\RR$ instead of on the undirected graph $\tilde{\RR}$. Nevertheless, the resulting object has the same Zariski closure as the Rauzy--Veech group, since the Zariski closure of a monoid is a group.
	\end{rem}
		
	\subsection{Rauzy--Veech groups in homotopy} In this section we follow the general discussion about Rauzy classes and Dehn twists by Avila, Matheus and Yoccoz \cite[Section 4]{AMY:hyperelliptic}.
	
	Consider a Rauzy class $\RR$. For every $\pi \in \RR$ there is a canonical way to obtain a translation surface from an irreducible permutation. Namely, we consider the translation surface with length data $\lambda_\alpha = 1$ and suspension data $\tau_\alpha = \pi_{\mathrm{b}}(\alpha) - \pi_{\mathrm{t}}(\alpha)$ for each $\alpha \in \A$. We denote the resulting polygon by $P_\pi \subseteq \C$, the resulting surface by $M_\pi$, its set of marked points by $\Sigma_\pi$ and the pure mapping class group of $M_\pi$ relative to $\Sigma_\pi$ by $\Mod(M_\pi, \Sigma_\pi)$. We equip $M_\pi$ with a basepoint $\ast_\pi = 1/2 \in \C$ and we set $O_\pi = d/2 \in \C$. We denote by $\Sigma_\pi^*$ the set consisting of $O_\pi$ and the midpoints of the sides of $P_\pi$.
	
	For each arrow $\gamma = \pi \to \pi'$ in $\RR$ there exists a homeomorphism $H_\gamma \colon M_\pi \to M_{\pi'}$ respecting the naming of the sets of marked points $\Sigma_\pi$, $\Sigma_{\pi'}$ \cite[Section 4.1.2]{AMY:hyperelliptic}. We denote by $[H_\gamma]$ its isotopy class from $(M_\pi, \Sigma_\pi \cup \Sigma_\pi^*)$ to $(M_\pi, \Sigma_{\pi'} \cup \Sigma_{\pi'}^*)$ relative to $\Sigma_\pi \cup \Sigma_{\pi}^*$. We also define $B_{\gamma^{-1}} = H_\gamma^{-1}$ for the arrow $\gamma^{-1} = \pi' \to \pi \in \tilde{\RR}$. Now consider a walk $\gamma = \gamma_1 \gamma_2 \dotsb \gamma_n$ in $\tilde{\RR}$ starting at $\pi$ and ending at $\pi'$. Similarly, we define $[H_\gamma] = [H_{\gamma_n}][H_{\gamma_{n-1}}] \dotsb [H_{\gamma_1}]$. We can then define the modular Rauzy--Veech group in an analogous way to the homological case:
	\begin{defn}
		Let $\RR$ be a Rauzy class and $\pi \in \RR$ be a fixed vertex. We define the \emph{modular Rauzy--Veech group} $\MRV(\pi)$ of $\pi$ as the set of mapping classes of the form $[H_\gamma] \in \Mod(M_\pi, \Sigma_\pi)$ where $\gamma$ is a cycle on $\tilde{\RR}$ with endpoints at $\pi$.
	\end{defn}
	
	The action of the map $H_\gamma$ on the fundamental groups, where $\gamma = \pi \to \pi'$ is an arrow in $\RR$, can be made explicit. Indeed, let $\alpha_{\mathrm{w}}$ and $\alpha_{\mathrm{l}}$ be the winner and loser, respectively, of the Rauzy induction. Let $\theta_\alpha$ be a simple closed curve starting from $\ast_\pi$, then going to the midpoint of the top $\alpha$-side of $P_\pi$ and then coming back to $\ast_\pi$, oriented upwards. It is easy to see that the homotopy classes of $\{\theta_\alpha\}_{\alpha \in \A}$ generate $\pi_1(M_\pi \setminus \Sigma_\pi, \ast_\pi)$ and that their homology classes are a basis of $H_1(M_\pi \setminus \Sigma_\pi)$. We call $\{\theta_\alpha'\}_{\alpha \in \A}$ the analogous curves for $\pi'$. Then, the induced homomorphism $\pi_1(\gamma) \colon \pi_1(M_\pi \setminus \Sigma_\pi, \ast_\pi) \to \pi_1(M_{\pi'} \setminus \Sigma_{\pi'}, \ast_{\pi'})$ of $H_\gamma$ on the fundamental groups is (up to homotopy): $\pi_1(\gamma)(\theta_\alpha) = \theta_\alpha'$ for every $\alpha \neq \alpha_{\mathrm{l}}$ and
	\[
		\pi_1(\gamma)(\theta_{\alpha_{\mathrm{l}}}) =
		\begin{cases}
			\theta_{\alpha_{\mathrm{l}}}' \star (\theta_{\alpha_{\mathrm{w}}}')^{-1} & \text{if $\gamma$ is of top type} \\
			(\theta_{\alpha_{\mathrm{w}}}')^{-1} \star \theta_{\alpha_{\mathrm{l}}}' & \text{if $\gamma$ is of bottom type.} \\
		\end{cases}
	\]
	The induced homomorphism of $H_\gamma$, written in terms of the homology classes of the curves $\{\theta_\alpha\}_{\alpha \in \A}$ and $\{ \theta_{\alpha}' \}_{\alpha \in \A}$, is exactly $\Id - E_{\alpha_{\mathrm{l}} \alpha_{\mathrm{w}}} = B_\gamma^{-1}$. Therefore, the induced action of $\MRV(\pi)$ on homology is precisely $\RV(\pi)$. We will identify the matrix $B_\gamma = \Id + E_{\alpha_{\mathrm{l}} \alpha_{\mathrm{w}}}$ with a map $H_1(M_{\pi'} \setminus \Sigma_{\pi'}) \to H_1(M_\pi \setminus \Sigma_\pi)$ by using these bases.

	\subsubsection{Dehn twists} For each $\alpha \in \A$, we have that at least one of the Dehn twists along $\theta_\alpha$ belongs to $\MRV(\pi)$. These Dehn twists will be useful to generate $\RV(\pi)$.
	\begin{lem}{\label{lem:dehntwists}}
		Let $\pi$ be a vertex of a Rauzy class $\RR$. Then, the left or right Dehn twist along $\theta_\alpha$ belongs to $\MRV(\pi)$ for every $\alpha \in \A$.
	\end{lem}
	\begin{proof}
		Fix $\alpha \in \A$ and let $\xi_\alpha$ be the isotopy class of $\theta_\alpha$. Let $\gamma$ be a (possibly empty) walk on $\RR$ starting at $\pi$ ending on a vertex $\pi' = (\pi_{\mathrm{t}}', \pi_{\mathrm{b}}')$ such that $\alpha = (\pi_{\mathrm{t}}')^{-1}(d)$ or $\alpha = (\pi_{\mathrm{b}}')^{-1}(d)$, and such that $\pi'$ is the first vertex of $\gamma$ with this property. Such a walk exists because the Rauzy classes are strongly connected.
				
		Let $\gamma'$ be the pure cycle at $\pi'$ having $\alpha$ as the winner. That is, all the arrows of $\gamma$ are different and their winner is $\alpha$. Assume that it is a pure cycle of top type. We have that $[H_{\gamma \gamma' \gamma^{-1}}]$ is equal to the left Dehn twist along $\theta_\alpha$, which we denote $T_{\xi_a}$. Indeed, let $\xi_\alpha'$ be the isotopy class of $\theta_\alpha'$. It easy to see that $[H_{\gamma'}] = T_{\xi_\alpha'}$, since it is computed along a pure cycle \cite[Section 4.1.6]{AMY:hyperelliptic}. We obtain that $[H_{\gamma \gamma' \gamma^{-1}}] = [H_\gamma^{-1}]T_{\xi_\alpha'}[H_\gamma] =  T_{[H_\gamma^{-1}](\xi_\alpha')}$, where the last equality is straightforward by definition of Dehn twist. Our hypothesis on $\gamma$ and the the previous discussion about the action of $H_\gamma$ on the fundamental groups show that $[H_\gamma]$ maps $\xi_\alpha$ to $\xi_\alpha'$, since $\alpha$ is not the loser of any arrow of $\gamma$. We conclude that $[H_{\gamma \gamma' \gamma^{-1}}] = T_{\xi_\alpha}$. If $\gamma'$ is a pure cycle of bottom type, similar computations show that $[H_{\gamma \gamma' \gamma^{-1}}]$ is equal to the right Dehn twist along $\theta_\alpha$.
	\end{proof}
	
	Finally, the induced action $T_\alpha \colon H_1(M_\pi \setminus \Sigma_\pi) \to H_1(M_\pi \setminus \Sigma_\pi)$ of the (either left or right) Dehn twist $T_{\xi_\alpha}$ has a simple expression in the basis consisting of the homology classes of $\{\theta_{\alpha}\}_{\alpha \in \A}$. Indeed, the alternate form $\Omega_\pi$ is defined as the intersection form of such curves, so we have that $T_\alpha(u) = u + \langle [\theta_\alpha], u\rangle [\theta_\alpha]$ for any $u \in H_1(M_\pi \setminus \Sigma_\pi)$, where $[\theta_\alpha]$ is the homology class of $\theta_\alpha$ and $\langle \cdot, \cdot \rangle$ is the bilinear form induced by $\Omega_\pi$. Since we will use the basis $\{[\theta_{\alpha}]\}_{\alpha \in \A}$, we identify $[\theta_\alpha]$ with the canonical vector $e_\alpha$ whose only non-zero coordinate, equal to $1$, is at position $\alpha$.
	
	\subsection{General properties of Rauzy--Veech groups}
	
	In this section we state some general properties of Rauzy--Veech groups. Recall that we consider the action on \emph{row} vectors. First, we state the definition of a symplectic transvection:
	\begin{defn}
		For a vector $v \in H_1(M_\pi \setminus \Sigma_\pi)$, we define the \emph{symplectic transvection} $T_v \in \Sp(\Omega_\pi, \Z)$ along $v$ as:
		\[
			T_v(u) = u + \langle v, u \rangle v.
		\]
	\end{defn}
	Observe that $T_{-v} = T_v$ for any $v \in H_1(M_\pi \setminus \Sigma_\pi)$.
	
	We have previously shown that $T_\alpha = T_{e_\alpha}$ for each $\alpha \in \A$. Conversely, for each primitive vector $v \in H_1(M_\pi \setminus \Sigma_\pi)$ we can choose a simple closed curve $c$ on $M_\pi \setminus \Sigma_\pi$ whose homology class is $v$. If $\xi$ is the isotopy class of $c$, then it is easy to see that the (either left or right) Dehn twist $T_{\xi} \in \Mod(M_\pi, \Sigma_\pi)$ acts in homology as $T_v$. 
	
	The following lemma allows us to construct more symplectic transvections from a set of generators.
	
	\begin{lem} \label{lem:generate_dehn_twists}
		Let $v, w \in H_1(M_\pi \setminus \Sigma_\pi)$ such that $\langle v, w \rangle = 1$. Then,
		\begin{itemize}
			\item $T_w^{-1} T_v T_w = T_v T_w T_v^{-1} = T_{v + w}$;
			\item $T_w T_v T_w^{-1} = T_v^{-1} T_w T_v = T_{v - w}$.
		\end{itemize}
	\end{lem}
	\begin{proof}
		It is a well-known fact that if $f \in \Mod(M_\pi, \Sigma_\pi)$, then $f T_\xi f^{-1} = T_{f(\xi)}$ \cite[Fact 3.7]{FM:primer}. The proof follows directly from this fact and the previous discussion, since the equality $\langle v, w\rangle = 1$ implies that $v$ and $w$ are primitive.
	\end{proof}
	
	We have that the group generated by symplectic transvections along canonical vectors is an invariant of the Rauzy class:
	\begin{lem} \label{lem:transvections_rauzy_class}
		Let $\pi$, $\pi'$ be vertices of a Rauzy class $\RR$. For $\alpha \in \A$, define $T_{e_\alpha} \in \RV(\pi)$ and $T_{e_\alpha'} \in \RV(\pi')$ as the symplectic transvections along the homology classes of $\theta_\alpha$ and $\theta_\alpha'$, respectively. Let $G \subseteq \RV(\pi)$ and $G' \subseteq \RV(\pi')$ be the subgroups generated by $\{T_{e_\alpha}\}_{\alpha \in \A}$ and $\{T_{e_\alpha'}\}_{\alpha \in \A}$, respectively. If $\gamma$ is any walk in $\tilde{\RR}$ joining $\pi$ and $\pi'$, then the isomorphism $\Sp(\Omega_\pi, \Z) \to \Sp(\Omega_{\pi'}, \Z)$ defined by $S \mapsto B_\gamma S B_{\gamma}^{-1}$ restricts to an isomorphism between $G$ and $G'$.
	\end{lem}
	\begin{proof}
		By induction, we can assume that $\gamma$ is the arrow $\pi \to \pi'$. It is enough to show that, for each $\alpha \in \A$, there exists $S \in G$ such that $B_\gamma S B_\gamma^{-1} = T_{e_\alpha'}$. Observe that $B_\gamma T_v B_\gamma^{-1} = T_{v B_\gamma^{-1}}$ for any $v \in H_1(M_\pi \setminus \Sigma_\pi)$. Therefore, if $\alpha \neq \alpha_{\mathrm{l}}$, we can choose $S = T_{e_\alpha}$ and, if $\alpha = \alpha_{\mathrm{l}}$, we can choose $S = T_{e_\alpha + e_{\alpha_{\mathrm{w}}}}$, which is equal to either $T_{e_{\alpha_{\mathrm{w}}}}^{-1} T_{e_\alpha} T_{e_{\alpha_{\mathrm{w}}}}$ or $T_{e_{\alpha_{\mathrm{w}}}} T_{e_\alpha} T_{e_{\alpha_{\mathrm{w}}}}^{-1}$ depending on the type of $\gamma$.
	\end{proof}
	We will end up showing that, in absolute homology, the group $G$ in the previous lemma is the entire Rauzy--Veech group.
	
	To formalize the idea of generating symplectic transvections from a set of generators using the previous lemma, we introduce the following definitions:
	\begin{defn} \label{def:omega_closed}
		For $V \subseteq H_1(M_\pi \setminus \Sigma_\pi)$, we consider the group $G_V \subseteq \Sp(\Omega_\pi, \Z)$ generated by the symplectic transvections $\{T_v\}_{v \in V}$. We say that a set $X \subseteq H_1(M_\pi \setminus \Sigma_\pi)$ is \emph{$\Omega_\pi$-closed} if $X = -X$ and $\langle v, w\rangle = 1$ implies $v+w,v-w \in X$ for every $v, w \in X$. We denote by $V^{\Omega_\pi}$ the smallest $\Omega_\pi$-closed set containing $V$. By the previous lemma, $\{T_v\}_{v \in V^{\Omega_\pi}} \subseteq G_V$.
	\end{defn}

	We also need a way to generate squares of symplectic transvections along sum of vectors when their intersection form is zero. We have the following:
	\begin{lem} \label{lem:generate_dehn_twists_disjoint}
		Let $v_1, v_2, v_3, v_4 \in H_1(M_\pi \setminus \Sigma_\pi)$ such that
		\[
			\langle v_i, v_j \rangle_{i, j = 1}^4 =
			\begin{pmatrix}
				0 & 0 & 0 & 1 \\
				0 & 0 & 1 & 1 \\
				0 & -1 & 0 & 1 \\
				-1 & -1 & - 1 & 0
			\end{pmatrix}
		\]
		then $T_{v_1 + v_2}^2, T_{v_1 - v_2}^2, T_{v_1 + v_3}^2, T_{v_1 - v_3}^2 \in G_{\{v_1,v_2,v_3,v_4\}}$.
	\end{lem}
	\begin{proof}
		We will first show that $2v_1 + v_2, 2v_1 - v_3 \in \{v_1,v_2,v_3,v_4\}^{\Omega_\pi}$. For this, it is enough to observe that:
		\begin{align*}
			1 = \langle v_1, v_4 \rangle &= \langle v_1 + v_4, -v_1 \rangle = \langle 2v_1 + v_4, -v_3 \rangle = \langle 2v_1 - v_3 + v_4, v_4 \rangle \\
			&= \langle 2v_1 - v_3, v_2 \rangle = \langle 2v_1 + v_2 - v_3, v_3 \rangle.
		\end{align*}
		We obtain that $T_{2v_1 + v_2}, T_{2v_1 - v_3} \in G_{\{v_1,v_2,v_3,v_4\}}$. Now, we have that:
		\[
			T_{2v_1 + v_2}(u) = u + 4 \langle v_1, u\rangle v_1 + 2 \langle v_1, u \rangle v_2 + 2\langle v_2, u \rangle v_1 + \langle v_2, u \rangle v_2
		\]
		and, since $\langle v_1, v_2 \rangle = \langle v_2, v_1 \rangle = 0$, we also have that:
		\[
			T_{v_1}^{-2} T_{v_2}T_{2v_1 + v_2}(u) = u + 2 \langle v_1, u\rangle v_1 + 2 \langle v_1, u \rangle v_2 + 2\langle v_2, u \rangle v_1 + 2\langle v_2, u \rangle v_2 = T_{v_1+v_2}^2(u).
		\]
		Similarly, we have that:
		\[
			T_{2v_1 - v_3}(u) = u + 4 \langle v_1, u\rangle v_1 - 2 \langle v_1, u \rangle v_3 - 2\langle v_3, u \rangle v_1 + \langle v_3, u \rangle v_3
		\]
		and that:
		\[
			T_{v_1}^{-2} T_{v_3} T_{2v_1 - v_3}(u) = u + 2 \langle v_1, u\rangle v_1 - 2 \langle v_1, u \rangle v_3 - 2\langle v_3, u \rangle v_1 + 2\langle v_3, u \rangle v_3 = T_{v_1-v_3}^2(u).
		\]
		We can prove that $T_{v_1-v_2}^2, T_{v_1+v_3}^2 \in G_{\{v_1,v_2,v_3,v_4\}}$ in a similar way by swapping the roles of $v_2$ and $v_3$.
	\end{proof}
	
	The signs of the intersections are not important as shown by the following useful corollary:
	\begin{cor} \label{cor:generate_dehn_twists_nosign}
		Let $v_1', v_2', v_3', v_4' \in H_1(M_\pi \setminus \Sigma_\pi)$ such that
		\[
		|\langle v_i', v_j' \rangle|_{i, j = 1}^4 =
		\begin{pmatrix}
		0 & 0 & 0 & 1 \\
		0 & 0 & 1 & 1 \\
		0 & 1 & 0 & 1 \\
		1 & 1 & 1 & 0
		\end{pmatrix}
		\]
		then $T_{v_1' + v_2'}^2, T_{v_1' - v_2'}^2, T_{v_1' + v_3'}^2, T_{v_1' - v_3'}^2 \in G_{\{v_1',v_2',v_3',v_4'\}}$.
	\end{cor}
	\begin{proof}
		We apply the previous lemma with $v_1 = \pm v_1'$, $\{v_2, v_3\} = \{\pm v_2', \pm v_3'\}$ and $v_4 = v_4'$, where we first choose the signs of $v_1, v_2$ and $v_3$ so $\langle v_1, v_4\rangle = \langle v_2, v_4 \rangle = \langle v_3, v_4 \rangle = 1$ and then we choose the order of $v_2$ and $v_3$ so $\langle v_2, v_3 \rangle = 1$. 
	\end{proof}
	
	Finally, we will prove that the Rauzy--Veech groups preserve the usual quadratic form $Q_\pi$ on $H_1(M_\pi \setminus \Sigma_\pi; \Z / 2\Z)$. We first recall the definition of this quadratic form \cites{J:spin}[Appendix C]{Z:representatives}: for any $u \in H_1(M_\pi \setminus \Sigma_\pi; \Z / 2\Z)$, we choose a smooth simple closed curve $c$ on $M_\pi \setminus \Sigma_\pi$ whose modulus-two homology class is $u$. We then define $Q_\pi(u) = \mathrm{ind}(c) + 1 \mod 2$, where $\mathrm{ind}(c)$ is the degree of the Gauss map (also known as index or turning number) of $c$. For every $\alpha \in \A$ we have that $Q_\pi(e_\alpha) = 1$ since the vector $e_\alpha$ is the homology class of the curve $\theta_\alpha$, which has index $0$.
	For a vector $u = \sum_{a \in \A} u_\alpha e_\alpha$ we then have that:
	\[
		Q_\pi(u) = \sum_{\alpha < \beta} u_\alpha (\Omega_\pi)_{\alpha \beta} u_\beta + \sum_{\alpha \in \A} u_\alpha \mod{2},
	\]
	where the notation ``$\sum_{\alpha < \beta}$'' means ``$\sum_{\pi_{\mathrm{t}}(\alpha) < \pi_{\mathrm{t}}(\beta)}$''.
	
	Throughout the entire article, we denote the modulus-two reduction of an object $\bullet$ by a bar above it: $\overline{\bullet}$. We can now prove that the Kontsevich-Zorich matrices preserve these quadratic forms. 
	\begin{lem} \label{lem:preserve_orthogonal}
		Let $\gamma = \pi \to \pi'$ be an arrow of some Rauzy class. Then, we have that $Q_{\pi'}(u) = Q_\pi(u \xoverline{B}_\gamma)$ for every $u \in H_1(M_{\pi'} \setminus \Sigma_{\pi'}; \Z / 2\Z)$.
	\end{lem}
	
	\begin{proof}
		Let $\alpha_{\mathrm{w}}$ and $\alpha_{\mathrm{l}}$ be the winner and loser, respectively, for the Rauzy induction. Then, $\xoverline{B}_\gamma = \xoverline{\Id} + \xoverline{E}_{\alpha_{\mathrm{l}} \alpha_{\mathrm{w}}}$. Thus,
		\[
			Q_{\pi}(u \xoverline{B}_\gamma) = Q_{\pi}(u + u_{\alpha_{\mathrm{l}}} \be_{\alpha_{\mathrm{w}}}) = Q_{\pi}(u) + u_{\alpha_{\mathrm{l}}} + u_{\alpha_{\mathrm{l}}} \langle u, \be_{\alpha_{\mathrm{w}}} \rangle_{\pi},
		\]
		where $\langle \cdot, \cdot \rangle_{\pi}$ denotes the bilinear form induced by $\Omega_{\pi}$.
		
		On the other hand, since $\Omega_{\pi'} = B_\gamma \Omega_\pi B_\gamma^\tr$ an immediate computation yields:
		\[
			(\Omega_{\pi'})_{\alpha \beta} = \begin{cases}
				\Omega_{\alpha \beta} & \alpha \neq \alpha_{\mathrm{l}}, \beta \neq \alpha_{\mathrm{l}} \\
				\Omega_{\alpha \beta} +\Omega_{\alpha_{\mathrm{w}}\beta} & \alpha = \alpha_{\mathrm{l}}, \beta \neq \alpha_{\mathrm{l}} \\
				\Omega_{\alpha \beta} + \Omega_{\alpha\alpha_{\mathrm{w}}} & \alpha \neq \alpha_{\mathrm{l}}, \beta = \alpha_{\mathrm{l}} \\
				0 & \alpha = \alpha_{\mathrm{l}} = \beta
			\end{cases}
		\]
		and, therefore,
		\begin{align*}
			Q_{\pi'}(u) &= \sum_{\alpha < \beta} u_\alpha (\bOmega_{\pi'})_{\alpha \beta} u_\beta + \sum_{\alpha \in \A} u_\alpha = \sum_{\alpha < \beta} u_\alpha (\bOmega_\pi)_{\alpha \beta} u_\beta + u_{\alpha_{\mathrm{l}}} \sum_{\alpha \neq \alpha_{\mathrm{l}}} u_\alpha (\bOmega_\pi)_{\alpha \alpha_{\mathrm{w}}} + \sum_{\alpha \in \A} u_\alpha \\
			&= \sum_{\alpha < \beta} u_\alpha (\bOmega_\pi)_{\alpha \beta} u_\beta + u_{\alpha_{\mathrm{l}}} + u_{\alpha_{\mathrm{l}}} \sum_{\alpha \in \A} u_\alpha (\bOmega_\pi)_{\alpha \alpha_{\mathrm{w}}} + \sum_{\alpha \in \A} u_\alpha \\
			&= Q_\pi(u) + u_{\alpha_{\mathrm{l}}} + u_{\alpha_{\mathrm{l}}} \langle u, \be_{\alpha_{\mathrm{w}}} \rangle_\pi = Q_\pi(u \xoverline{B}_\gamma)
		\end{align*}
		where we used that $\bOmega_{\alpha_{\mathrm{l}} \alpha_{\mathrm{w}}} = 1$.
	\end{proof}
	
	We obtain the following straightforward corollary:
	
	\begin{cor}
		The action on $H_1(M_\pi \setminus \Sigma_\pi; \Z / 2\Z)$ of the Rauzy--Veech group preserves $Q_\pi$. In other words, $\bRV(\pi) \subseteq O(Q_\pi)$, where $O(Q_\pi) \subseteq \Sp(\bOmega_\pi, \Z / 2\Z)$ is the orthogonal group induced by $Q_\pi$.
	\end{cor}
	
	We will show that, for some connected components of strata, the previous corollary characterizes the Rauzy--Veech group. That is, the group is the preimage of $O(Q_\pi)$ for the modulus-two reduction. Nevertheless, this is not true for every connected component. In particular, for $g \geq 3$ and every hyperelliptic connected component of $\H(2g-2)$ or $\H(g-1,g-1)$ there exist orthogonal transvections (see \Cref{sec:orthogonal}) not belonging to the Rauzy--Veech group.
	
	\section{Generating the level-two congruence subgroup} \label{sec:leveltwo}

	The first part of our proof that the Rauzy--Veech group of minimal strata is the full preimage of $O(Q)$ consists of showing that $\ker(\Sp(\Omega_\pi, \Z) \to \Sp(\bOmega_\pi, \Z / 2\Z)) \subseteq \RV(\pi)$  where $\pi$ represents a non-hyperelliptic connected component of $\H(2g - 2)$. That is, we have to prove that the level-two congruence subgroup is contained in $\RV(\pi)$. 
	
	For this section, we will use the following explicit permutation representatives of minimal strata computed by Zorich \cite[Proposition 3, Proposition 4]{Z:representatives}:
	{\small\[
		\tau^{(g)} =
		\begin{pmatrix}
			0 & 1 & 2 & 3 & 5 & 6 & \cdots & 3g - 7 & 3g - 6 & 3g - 4 & 3g - 3 \\
			3 & 2 & 6 & 5 & 9 & 8 & \cdots & 3g - 3 & 3g - 4 & 1 & 0
		\end{pmatrix}
	\]}
	for $g \geq 3$ and
	{\small\[
		\sigma^{(g)} =
		\begin{pmatrix}
			0 & 1 & 2 & 3 & 5 & 6 & \cdots & 3g - 7 & 3g - 6 & 3g - 4 & 3g - 3 \\
			6 & 5 & 3 & 2 & 9 & 8 & \cdots & 3g - 3 & 3g - 4 & 1 & 0
		\end{pmatrix}
	\]}
	for $g \geq 4$. One has that $M_{\tau^{(g)}} \in \H(2g - 2)^{\odd}$ and that $M_{\sigma^{(g)}} \in \H(2g - 2)^{\even}$.
	
	A finite set of generators of the level-two congruence subgroup is the following: given a symplectic basis $(b_\alpha)_{\alpha \in \A}$, the squares of symplectic transvections $T_{b_\alpha}^2$ and $T_{b_\alpha + b_\beta}^2$ for every $\alpha, \beta \in \A$ \cite[Appendix to Section 5]{M:tata}. Observe that this condition is redundant if $\alpha = \beta$, since $T_{2b_\alpha}^2 = (T_{b_\alpha}^2)^4$. We will prove this fact for the family $(\tau^{(g)})_{g \geq 3}$ and $(\sigma^{(g)})_{g \geq 4}$ separately.
	
	\subsection{Proof for the first family} Our objective is now proving that the level-two congruence subgroup can be generated for the family $(\tau^{(g)})_{g \geq 3}$ of irreducible permutations. We will denote $M^{(g)}$ for $M_{\tau^{(g)}}$, $\Sigma^{(g)}$ for $\Sigma_{\tau^{(g)}}$ and $\Omega^{(g)}$ for $\Omega_{\tau^{(g)}}$ until explicitly stated.
	
	We start by finding an appropriate symplectic basis:
	\begin{lem} \label{lem:first_symplectic}
		If $g \geq 3$, a symplectic basis for $(H_1(M^{(g)} \setminus \Sigma^{(g)}), \langle \cdot, \cdot \rangle)$ is given by the following pairs of vectors:
		\begin{align*}
			&\qquad\{e_2, e_3\}, \{e_5, e_6\}, \{e_8, e_9\}, \dotsc, \{ e_{3g - 4}, e_{3g - 3} \}, \\
			&\{ e_0 + (-e_2 + e_3) + (-e_5 + e_6) + \dotsb + (-e_{3g-4} + e_{3g-3}), \\
			&\ e_1 + (-e_2 + e_3) + (-e_5 + e_6) + \dotsb + (-e_{3g-4} + e_{3g-3}) \}.
		\end{align*}
	\end{lem}
	\begin{proof}
		The proof is a straightforward computation.
	\end{proof}

	We denote the last two vectors in the previous lemma by $v^*$ and $w^*$, respectively. The next series of lemmas show that $\RV({\tau^{(g)}})$ contains the desired squares of symplectic transvections for this basis. Recall that if $v \in H_1(M^{(g)} \setminus \Sigma^{(g)})$ belongs to $\{e_\alpha\}_{\alpha \in \A}^{\Omega^{(g)}}$, then $T_v \in \RV(\tau^{(d)})$ (and, therefore, $T_v^2 \in \RV(\tau^{(g)})$) by \Cref{lem:dehntwists} and \Cref{lem:generate_dehn_twists}. For this reason, we will sometimes prove that $v \in \{e_\alpha\}_{\alpha \in \A}^{\Omega^{(g)}}$. Moreover, if two vectors $v, w$ belong to $\{e_\alpha\}_{\alpha \in \A}^{\Omega^{(g)}}$ and satisfy $\langle v, w \rangle = 1$, then $v + w$ belongs to $\{e_\alpha\}_{\alpha \in \A}^{\Omega^{(g)}}$, so in these cases we get that $T_{v+w}^2 \in \RV(\tau^{(d)})$ ``for free''. This is the case for the vectors $e_2, e_3, e_5, e_6, \dotsc, e_{3g-4}, e_{3g-3}$ and also for $e_2+e_3, e_5+e_6, \dotsc, e_{3g-4}+e_{3g-3}$.
	
	The following lemma is not a strict part of the proof, but will allow us to simplify many arguments.
	\begin{lem} \label{lem:alt_sum_1}
		If $\beta \mod{6} = 3$, then $(-e_2 + e_3) + (-e_5 + e_6) + \dotsb + (-e_{\beta-1} + e_\beta) \in \{e_\alpha\}_{\alpha \in \A}^{\Omega^{(g)}}$.
	\end{lem}
	\begin{proof}
		Observe that $v_0 = -e_2 + e_3 \in \{e_\alpha\}_{\alpha \in \A}^{\Omega^{(g)}}$. We continue inductively: for $k \geq 1$ we have that:
		\begin{align*}
			1 &= \langle e_0, e_{6k+3} \rangle = \langle -e_0 + e_{6k+3}, -e_{6k+5} \rangle = \langle -e_0 + e_{6k+3} - e_{6k+5}, -e_{6k+8}\rangle\\
			&= \langle -e_0 + e_{6k+3} - e_{6k+5} -e_{6k+8}, e_0 \rangle 
			= \langle e_{6k+3} - e_{6k+5} -e_{6k+8}, v_k \rangle \\
			&= \langle v_k + e_{6k+3} - e_{6k+5} -e_{6k+8}, -e_{6k + 3} \rangle = \langle v_k - e_{6k+5} -e_{6k+8}, -e_{6k + 6} \rangle \\
			&= \langle v_k - e_{6k+5} +e_{6k + 6} -e_{6k+8}, -e_{6k+9}  \rangle
		\end{align*}
		so $v_{k+1} = v_k + (-e_{6k+5} + e_{6k+6}) + (-e_{6k+8} + e_{6k+9}) \in \{e_\alpha\}_{\alpha \in \A}^{\Omega^{(g)}}$. By continuing this way until $k = (\beta - 3)/6$ we obtain the desired result.
	\end{proof}
	
	\begin{lem} \label{lem:first_zero_intersection}
		Let $\alpha, \beta \in \{2, 3, 5, 6, \dotsc, 3g-4,3g-3\}$ such that $\langle e_\alpha, e_\beta \rangle = 0$. Then, we have that $T_{e_\alpha + e_\beta}^2 \in \RV(\tau^{(g)})$.
	\end{lem}
	\begin{proof}
		Let $\beta' \in \A$ such that $|\langle e_\beta, e_{\beta'}\rangle| = 1$. We can use \Cref{cor:generate_dehn_twists_nosign} with $v_1' = e_\alpha$, $v_2' = e_\beta$, $v_3' = e_{\beta'}$, $v_4' = e_0$.
	\end{proof}
	
	\begin{lem} \label{lem:v_s_and_w_s}
		If $3g - 3 \mod{6} = 3$, then $T_{v^*}^2, T_{w^*}^2 \in \RV(\tau^{(g)})$ and $v^* + w^* \in \{e_\alpha\}_{\alpha \in \A}^{\Omega^{(g)}}$.
		
		If $3g - 3 \mod{6} = 0$, then $v^*, w^*, v^* + w^* \in \{e_\alpha\}_{\alpha \in \A}^{\Omega^{(g)}}$.
	\end{lem}
	\begin{proof}
		If $3g - 3 \mod{6} = 3$, then $v = (-e_2 + e_3) + \dotsb + (-e_{3g-4} + e_{3g - 3}) \in \{e_\alpha\}_{\alpha \in \A}^{\Omega^{(g)}}$ by \Cref{lem:alt_sum_1}. We can use \Cref{cor:generate_dehn_twists_nosign} with $v_1' = v$, $v_2' = e_0$, $v_3' = e_1$ and $v_4' = e_2$ to show that $T_{v^*}^2, T_{w^*}^2 \in \RV(\tau^{(g)})$. Moreover, we have that $v^* + w^* \in \{e_\alpha\}_{\alpha \in \A}^{\Omega^{(g)}}$ since
		\begin{align*}
			1 &= \langle e_0, e_{3g - 3} \rangle = \langle e_0 - e_{3g - 3}, -v\rangle = \langle e_0 - e_{3g - 3} + v, -v\rangle \\
			&= \langle e_0 - e_{3g - 3} + 2v, -e_{3g - 3}\rangle = \langle e_0 + 2v, e_1\rangle
		\end{align*}
		so $v^* + w^* = e_0 + e_1 + 2v \in \{e_\alpha\}_{\alpha \in \A}^{\Omega^{(g)}}$.
		
		If $3g - 3 \mod{6} = 0$, we have that $v = (-e_2 + e_3) + \dotsb + (-e_{3g-7} + e_{3g-6}) \in \{e_\alpha\}_{\alpha \in \A}^{\Omega^{(g)}}$ by \Cref{lem:alt_sum_1}. Observe that:
		\begin{align*}
			1 &= \langle e_1, e_{3g-6} \rangle = \langle e_1 - e_{3g-6}, e_{3g-4} \rangle = \langle e_1 - e_{3g-6} - e_{3g-4}, e_0 \rangle \\
			&= \langle e_0 + e_1 - e_{3g-6} - e_{3g-4}, -v \rangle = \langle e_0 + e_1 - e_{3g-6} - e_{3g-4}  + v, e_{3g-6} \rangle \\
			&= \langle e_0 + e_1 - e_{3g-4}  + v, e_{3g-3} \rangle,
		\end{align*}
		so $w = v + e_0 + e_1 + (-e_{3g - 4} + e_{3g - 3}) \in \{e_\alpha\}_{\alpha \in \A}^{\Omega^{(g)}}$. Since $\langle w, -e_0 \rangle = 1$ and $\langle w, e_1 \rangle = 1$, we obtain that $v^* = w - e_1$ and $w^* = w - e_0$ belong to $\{e_\alpha\}_{\alpha \in \A}^{\Omega^{(g)}}$. We also obtain that $v^* + w^* \in \{e_\alpha\}_{\alpha \in \A}^{\Omega^{(g)}}$ since $\langle v^*, w^* \rangle = 1$.
	\end{proof}
	
	The next lemma completes the proof:
	\begin{lem} \label{lem:other_sum}
		Let $\beta \in \{2, 3, 5, 6, \dotsc, 3g-4,3g-3\}$. Then, $T_{v^* + e_\beta}^2, T_{w^* + e_\beta}^2 \in \RV(\tau^{(g)})$.
	\end{lem}
	\begin{proof}
		We consider two cases:
	
		If $3g - 3 \mod{6} = 3$, we have that $v = (-e_2 + e_3) + \dotsb + (-e_{3g-4} + e_{3g - 3}) \in \{e_\alpha\}_{\alpha \in \A}^{\Omega^{(g)}}$ by \Cref{lem:alt_sum_1}. Observe that $v^* + e_\beta \in \{e_\alpha\}_{\alpha \in \A}^{\Omega^{(g)}}$ since $1 = \langle e_0, e_\beta \rangle = \langle e_0 + e_\beta, v\rangle$. We can prove that $w^* + e_\beta \in \{e_\alpha\}_{\alpha \in \A}^{\Omega^{(g)}}$ analogously.
		
		If $3g - 3 \mod{6} = 0$, let $\beta' \in \A$ such that $|\langle e_\beta, e_{\beta'} \rangle| = 1$. We have that $v^*, w^* \in \{e_\alpha\}_{\alpha \in \A}^{\Omega^{(g)}}$ by the previous lemma. We conclude by \Cref{cor:generate_dehn_twists_nosign} with $v_1' = v^*$, $v_2' = e_\beta$, $v_3' = e_{\beta'}$ and $v_4' = e_1$ and with $v_1' = w^*$, $v_2' = e_\beta$, $v_3' = e_{\beta'}$ and $v_4' = e_0$.
	\end{proof}
	
	We obtain the following proposition:
	\begin{prop}\label{lem:kernel1}
		For any $g \geq 3$, the symplectic transvections $\{T_\alpha\}_{\alpha \in \A}$ generate the level-two congruence subgroup $\ker(\Sp(\Omega_{\tau^{(g)}}, \Z) \to \Sp(\bOmega_{\tau^{(g)}}, \Z / 2\Z))$. In particular, it is contained in $\RV(\tau^{(g)})$.
	\end{prop}
	
	\subsection{Proof for the second family} We will now prove that the level-two congruence subgroup can be generated for the family $(\sigma^{(g)})_{g \geq 4}$ of irreducible permutations. The proof is very similar to the one for the first family, although some computations are different so it is necessary to present it in full detail. We will denote $M^{(g)}$ for $M_{\sigma^{(g)}}$, $\Sigma^{(g)}$ for $\Sigma_{\sigma^{(g)}}$ and $\Omega^{(g)}$ for $\Omega_{\sigma^{(g)}}$ until explicitly stated.
	
	We start by finding an appropriate symplectic basis:
	\begin{lem} \label{lem:second_symplectic}
		If $g \geq 4$, a symplectic basis for $(H_1(M^{(g)} \setminus \Sigma^{(g)}), \langle \cdot, \cdot \rangle)$ is given by the following pairs of vectors:
		\begin{align*}
			&\{e_2 - e_5 + e_6\}, \{e_3 - e_5 + e_6\}, \{e_5, e_6\}, \{e_8, e_9\}, \dotsc, \{ e_{3g - 4}, e_{3g - 3} \}, \\
			&\qquad\!\{ e_0 + (-e_2 + e_3) + (-e_5 + e_6) + \dotsb + (-e_{3g-4} + e_{3g-3}), \\
			&\qquad e_1 + (-e_2 + e_3) + (-e_5 + e_6) + \dotsb + (-e_{3g-4} + e_{3g-3}) \}.
		\end{align*}
	\end{lem}
	\begin{proof}
		The proof is a straightforward computation.
	\end{proof}
	
	As before, we denote the last two vectors by $v^*$ and $w^*$ respectively. Observe that we have that $e_5, e_6, \dotsc, e_{3g-4}, e_{3g-3}$ and $e_5+e_6, \dotsc, e_{3g-4}+e_{3g-3}$ belong to $\{e_\alpha\}_{\alpha \in \A}^{\Omega^{(g)}}$ ``for free''.
	
	We have a lemma analogous to \Cref{lem:alt_sum_1}:
	\begin{lem} \label{lem:alt_sum_2}
		If $\beta \mod{6} = 3$, then $(-e_2 + e_3) + (-e_5 + e_6) + \dotsb + (-e_{\beta-1} + e_\beta) \in \{e_\alpha\}_{\alpha \in \A}^{\Omega^{(g)}}$.
	\end{lem}
	\begin{proof}
		We will prove that $(-e_2 + e_3) + (-e_5 + e_6) + (-e_8 + e_9) \in \{e_\alpha\}_{\alpha \in \A}^{\Omega^{(g)}}$. Then, we can continue inductively as in the proof of \Cref{lem:alt_sum_1}. We have that:
		\begin{align*}
			1 &= \langle e_0, e_3 \rangle = \langle e_0 + e_3, e_8 \rangle = \langle e_0 + e_3 - e_8, e_1 \rangle = \langle e_0 - e_1 + e_3 - e_8, -e_2 \rangle \\
			& = \langle e_0 - e_1 - e_2 + e_3 - e_8, -e_9 \rangle = \langle e_0 - e_1 - e_2 + e_3 - e_8 + e_9, e_0 \rangle \\
			& = \langle - e_1 - e_2 + e_3 - e_8 + e_9, -e_5 \rangle = \langle - e_1 - e_2 + e_3 - e_5 - e_8 + e_9, e_1 \rangle  \\
			&= \langle - e_2 + e_3 - e_5 - e_8 + e_9, -e_6 \rangle,
		\end{align*}
		which shows what we wanted.
	\end{proof}
	The proofs of some lemmas are very similar to the ones for the first family (replacing \Cref{lem:alt_sum_1} with \Cref{lem:alt_sum_2}) and in these cases we will just refer to such proofs. Indeed, one has that, for any $\alpha \in \{0, 1, 8, 9, 11, 12, \dotsc, 3g-3\}$, $\langle v, e_\alpha \rangle_{\tau^{(g)}} = \langle v, e_\alpha \rangle_{\sigma^{(g)}}$ for any $v$, so $\tau^{(g)}$ and $\sigma^{(g)}$ can only differ for $\alpha \in \{2, 3, 5, 6\}$ (and sometimes they are equal even in this case). We will implicitly exploit this fact to avoid having to repeat the proofs.
	
	\begin{lem} 
		Let $\alpha, \beta \in \{5, 6, \dotsc, 3g-4,3g-3\}$ such that $\langle e_\alpha, e_\beta \rangle = 0$. Then, we have that $T_{e_\alpha + e_\beta}^2 \in \RV(\sigma^{(g)})$.
	\end{lem}
	\begin{proof}
		The proof is identical to that of \Cref{lem:first_zero_intersection}.
	\end{proof}
	
	\begin{lem}
		If $3g - 3 \mod{6} = 3$, then $T_{v^*}^2, T_{w^*}^2 \in \RV(\sigma^{(g)})$ and $v^* + w^* \in \{e_\alpha\}_{\alpha \in \A}^{\Omega^{(g)}}$.
		
		If $3g - 3 \mod{6} = 0$, then $v^*, w^*, v^* + w^* \in \{e_\alpha\}_{\alpha \in \A}^{\Omega^{(g)}}$.
	\end{lem}
	\begin{proof}
		The proof is identical to that of \Cref{lem:v_s_and_w_s}.
	\end{proof}
	
	\begin{lem}
		Let $\beta \in \{5, 6, \dotsc, 3g-4,3g-3\}$. Then, $T_{v^* + e_\beta}^2, T_{w^* + e_\beta}^2 \in \RV(\sigma^{(g)})$.
	\end{lem}
	\begin{proof}
		 The proof is identical to that of \Cref{lem:other_sum}.
	\end{proof}
	
	\begin{lem}
		We have that $T_{e_\beta - e_5 + e_6}^2 \in \RV(\sigma^{(g)})$ and $(e_2 - e_5 + e_6) + (e_3 - e_5 + e_6) \in \{e_\alpha\}_{\alpha \in \A}^{\Omega^{(g)}}$ for any $\beta \in \{2, 3\}$.
	\end{lem}
	\begin{proof}
		Observe that $T_{e_2 - e_5 + e_6}^2, T_{e_3 - e_5 + e_6}^2 \in \RV(\sigma^{(g)})$ by \Cref{cor:generate_dehn_twists_nosign} with $v_1' = e_5 - e_6$, $v_2' = e_2$, $v_3' = e_3$ and $v_4' = e_5$. Moreover, we have that:
		\[
			1 = \langle e_2, e_5 \rangle = \langle e_2 - e_5, e_5 \rangle = \langle e_2 - 2e_5, -e_6 \rangle = \langle e_2 - 2e_5 + e_6, -e_6 \rangle = \langle e_2 - 2e_5 + 2e_6, e_3 \rangle,
		\]
		so $(e_2 - e_5 + e_6) + (e_3 - e_5 + e_6) \in \{e_\alpha\}_{\alpha \in \A}^{\Omega^{(g)}}$.
	\end{proof}
	The following lemma completes the proof:
	\begin{lem}
		We have that $(e_{\beta} - e_5 + e_6) + e_{\alpha'} \in \{e_\alpha\}_{\alpha \in \A}^{\Omega^{(g)}}$ for any $\beta \in \{2, 3\}$ and any $\alpha' \in \{5,6,\dotsc, 3g-4,3g-3\}$. Moreover, $T_{(e_{\beta} - e_5 + e_6) + v^*}^2, T_{(e_{\beta} - e_5 + e_6) + w^*}^2 \in \RV(\sigma^{(g)})$.
	\end{lem}
	\begin{proof}
		If $\alpha' = 5$, the result is obvious. If $\alpha' = 6$, it is enough to observe that
		\[
			1 = \langle e_{\beta}, e_6 \rangle = \langle e_{\beta} + e_6, e_6 \rangle = \langle e_{\beta} + 2e_6, -e_5 \rangle.
		\]
		If $\alpha' \in \{8,9, \dotsc, 3g - 4, 3g - 3\}$, let $\beta' \in \A$ such that $\varepsilon = \langle e_{\alpha'}, e_{\beta'} \rangle \neq 0$. If $\varepsilon = 1$ we can use that:
		\begin{align*}
			1 &= \langle e_1, e_{\beta} \rangle = \langle -e_1 + e_{\beta}, -e_{\alpha'} \rangle = \langle -e_1 + e_{\beta} + e_{\alpha'}, -e_0 \rangle = \langle -e_0 - e_1 + e_{\beta} + e_{\alpha'}, -e_5 \rangle \\
			&= \langle -e_0 - e_1 + e_{\beta} - e_5 + e_{\alpha'}, - e_{\beta'} \rangle = \langle -e_0 - e_1 + e_{\beta} - e_5 + e_{\alpha'} - e_{\beta'}, e_0 \rangle \\
			&= \langle - e_1 + e_{\beta} - e_5 + e_{\alpha'} - e_{\beta'}, -e_6 \rangle = \langle - e_1 + e_{\beta} - e_5 + e_6 + e_{\alpha'} - e_{\beta'}, -e_1 \rangle \\
			&= \langle e_{\beta} - e_5 + e_6 + e_{\alpha'} - e_{\beta'}, e_{\beta'} \rangle
		\end{align*}
		to show that $(e_{\beta} - e_5 + e_6) + e_{\alpha'} \in \{e_\alpha\}_{\alpha \in \A}^{\Omega^{(g)}}$. The case $\varepsilon = -1$ can be proved similarly.
		
		If $3g - 3 \mod{6} = 3$, let $v = (-e_2 + e_3) + (-e_5 + e_6) + \dotsb + (e_{3g - 4} - e_{3g - 3}) \in \{e_\alpha\}_{\alpha \in \A}^{\Omega^{(g)}}$ by \Cref{lem:alt_sum_2}. We have that $1 = \langle e_\beta, v\rangle = \langle e_\beta + v, -e_0 \rangle$, so $v^* + e_\beta \in \{e_\alpha\}_{\alpha \in \A}^{\Omega^{(g)}}$. We can use \Cref{cor:generate_dehn_twists_nosign} with $v_1' = -e_5 + e_6$, $v_2' = v^* + e_2$, $v_3' = v^* + e_3$ and $v_4' = e_1 + e_5$ to conclude that $T_{(e_{\beta} - e_5 + e_6) + v^*}^2 \in \RV(\sigma^{(g)})$. The fact that $T_{(e_{\beta} - e_5 + e_6) + w^*}^2 \in \RV(\sigma^{(g)})$ can be proved similarly.
		
		Finally, if $3g - 3 \mod{6} = 0$ we have that $v^*, w^* \in \{e_\alpha\}_{\alpha \in \A}^{\Omega^{(g)}}$. Let $v \in \{v^*, w^*\}$. We will show that $(e_\beta - e_5 + e_6) + v \in \{e_\alpha\}_{\alpha \in \A}^{\Omega^{(g)}}$. If $v = v^*$, we put $\varepsilon = 1$, $w = e_1$ and $w' = e_0$. If $v = w^*$, we put $\varepsilon = -1$, $w = e_0$ and $w' = e_1$. Observe that $\langle v, w \rangle = \varepsilon$, that $\langle v, w' \rangle = 0$ and that:
		\begin{align*}
			1 &= \langle e_0, e_\beta \rangle = \langle -e_0 + e_\beta, -e_{3g - 3} \rangle = \langle -e_0 + e_\beta - e_{3g - 3}, -e_1 \rangle \\
			&= \langle -e_0 - e_1 + e_\beta - e_{3g - 3}, -e_5 \rangle = \langle -e_0 - e_1 + e_\beta - e_5 -  e_{3g - 3}, \varepsilon v \rangle \\
			&= \langle -e_0 - e_1 + e_\beta - e_5 -  e_{3g - 3} + v, w \rangle = \langle -e_0 - e_1 + e_\beta - e_5 -  e_{3g - 3} + v + w, -e_6 \rangle \\
			&= \langle -e_0 - e_1 + e_\beta - e_5 + e_6 -  e_{3g - 3} + v + w, -e_{3g - 3} \rangle \\
			&= \langle -e_0 - e_1 + e_\beta - e_5 + e_6 + v + w, -w' \rangle.
		\end{align*}
		Since $w + w' - e_0 - e_1 = 0$ in any case, we obtain that $e_\beta - e_5 + e_6 + v \in \{e_\alpha\}_{\alpha \in \A}^{\Omega^{(g)}}$.
	\end{proof}
	
	We obtain the following proposition:
	\begin{prop}\label{lem:kernel2}
		For any $g \geq 4$, the symplectic transvections $\{T_\alpha\}_{\alpha \in \A}$ generate the level-two congruence subgroup $\ker(\Sp(\Omega_{\sigma^{(g)}}, \Z) \to \Sp(\bOmega_{\sigma^{(g)}}, \Z / 2\Z))$. In particular, it is contained in $\RV(\sigma^{(g)})$.
	\end{prop}
	
	\section{Generating the orthogonal group} \label{sec:orthogonal}
	
	For this section and the next, we will use some other permutations representing the even and odd components of minimal strata. Indeed, we consider the following two families of irreducible permutations:
	{\small\[
		\tau^{(d)} =
		\begin{pmatrix}
			1 & 2 & 3 & \cdots & d-5 & d-4 & d-3 & d-2 & d-1 & d \\
			d & d-1 & d-2 & \cdots & 6 & 3 & 2 & 5 & 4 & 1
		\end{pmatrix}
	\]}
	for $d \geq 6$ and
	{\small\[
		\sigma^{(d)} =
		\begin{pmatrix}
			1 & 2 & 3 & \cdots & d-7 & d-6 & d-5 & d-4 & d-3 & d-2 & d-1 & d \\
			d & d-1 & d-2 & \cdots & 8 & 3 & 2 & 7 & 6 & 5 & 4 & 1
		\end{pmatrix}
	\]}
	for $d \geq 8$.
	
	These two families of permutations are similar in many ways. For the sake of simplicity and brevity, for this section and the next we write $\pi^{(d)}$ to refer to either $\tau^{(d)}$ or $\sigma^{(d)}$ and we set $M^{(d)} = M_{\pi^{(d)}}$, $\Sigma^{(d)} = \Sigma_{\pi^{(d)}}$, $\Omega^{(d)} = \Omega_{\pi^{(d)}}$ and $Q^{(d)} = Q_{\pi^{(d)}}$.
	
	We have that the genus $g_d$ of $M^{(d)}$ is $d/2$ if $d$ is even and $(d-1)/2$ if $d$ is odd. In the latter case, the kernel of the symplectic form is generated by $e_\sharp = (1,-1,1,\dotsc,-1,1)$.
	
	Moreover, using the notation from the classification of connected components \cite{KZ:connected_components}, we have that:
	\begin{lem} \label{lem:representatives}
		We have that for every $d \geq 6$:
		{\footnotesize \[
			M_{\tau^{(d)}} \in
			\begin{cases}
				\H(2g_d-2)^\odd & d\mod{8} \in \{0,6\} \\
				\H(g_d-1,g_d-1)^\mathrm{nonhyp}& d\mod{8} \in \{1,5\} \\
				\H(2g_d-2)^\even & d\mod{8} \in \{2,4\} \\
				\H(g_d-1,g_d-1)^\even& d\mod{8} = 3 \\
				\H(g_d-1,g_d-1)^\odd& d\mod{8} = 7 \\
			\end{cases}
		\]}
		and for every $d \geq 8$:
		{\footnotesize \[
			M_{\sigma^{(d)}} \in
			\begin{cases}
				\H(2g_d-2)^\even & d\mod{8} \in \{0,6\} \\
				\H(g_d-1,g_d-1)^\mathrm{nonhyp}& d\mod{8} \in \{1,5\} \\
				\H(2g_d-2)^\odd & d\mod{8} \in \{2,4\} \\
				\H(g_d-1,g_d-1)^\odd& d\mod{8} = 3 \\
				\H(g_d-1,g_d-1)^\even& d\mod{8} = 7. \\
			\end{cases}
		\]}
	\end{lem}
	
	\begin{proof}
		Observe that if $d \mod 8 \in \{1, 5\}$, then neither $M_{\tau^{(d)}}$ nor $M_{\sigma^{(d)}}$ are hyperelliptic, so they belong to the only non-hyperelliptic connected component of the stratum.
		
		For the rest of the cases, we can compute the Arf invariant of $Q^{(d)}$ to determine the connected component in which $M^{(d)}$ lies. First, it is easy to see that if $d$ is odd, then the Arf invariant of $Q^{(d)}$ is the same as that of $Q^{(d-1)}$, since any symplectic basis in $(d-1)$ dimensions is a maximal symplectic set in $d$ dimensions. Therefore, we only need to compute the Arf invariant for even $d$.
		
		We will use the definition of the Arf invariant as the value assumed most often by the quadratic form. We can count the number of non-singular vectors by establishing an appropriate recurrence. To this end, we define the following quantities:
		\begin{itemize}
			\item the number $|\mathrm{NS}_0(Q^{(d)})|$ of non-singular vectors with an even number of $1$s;
			\item the number $|\mathrm{NS}_1(Q^{(d)})|$ of non-singular vectors with an odd number of $1$s;
			\item the number $|\mathrm{S}_0(Q^{(d)})|$ of singular vectors with an even number of $1$s;
			\item the number $|\mathrm{S}_1(Q^{(d)})|$ of singular vectors with an odd number of $1$s.
		\end{itemize}
		Then, we have the following recurrence:
		\begin{align*}
			|\mathrm{NS}_0(Q^{(d+1)})| = |\mathrm{NS}_0(Q^{(d)})| + |\mathrm{NS}_1(Q^{(d)})|&, \quad
			|\mathrm{NS}_1(Q^{(d+1)})| = |\mathrm{NS}_1(Q^{(d)})| + |\mathrm{S}_0(Q^{(d)})| \\
			|\mathrm{S}_0(Q^{(d+1)})| = |\mathrm{S}_0(Q^{(d)}) + |\mathrm{S}_1(Q^{(d)})|&, \quad
			|\mathrm{S}_1(Q^{(d+1)})| = |\mathrm{S}_1(Q^{(d)}) + |\mathrm{NS}_0(Q^{(d)})|.
		\end{align*}
		Indeed, given a non-singular vector of $Q^{(d)}$ with an even number of $1$s we construct a non-singular vector of $Q^{(d+1)}$ with an even number of $1$s by setting the last coordinate to be zero. Moreover, for any non-singular vector of $Q^{(d)}$ with an odd number of $1$s we construct a non-singular vector of $Q^{(d+1)}$ with an even number of $1$s by setting the last coordinate to be $1$. This proves that $|\mathrm{NS}_0(Q^{(d+1)})| = |\mathrm{NS}_0(Q^{(d)})| + |\mathrm{NS}_1(Q^{(d)})|$, and the other relations can be shown similarly.
		
		We now assume $\pi^{(d)} = \tau^{(d)}$. We have the following base cases for the recurrence (see the \nameref{sec:appendix} for computations): $|\mathrm{NS}_0(Q^{(6)})| = 16$, $|\mathrm{NS}_1(Q^{(6)})| = 20$, $|\mathrm{S}_0(Q^{(6)})| = 16$ and $|\mathrm{S}_1(Q^{(6)})| = 12$. The linear recurrence can be then solved, yielding that:
		\begin{align*}
			|\mathrm{NS}_0(Q^{(d)})| = 2^{d-2} + 2^{(d-2)/2} \cos\left( \frac{d \pi}{4} \right)&, \quad 
			|\mathrm{NS}_1(Q^{(d)})| = 2^{d-2} - 2^{(d-2)/2} \sin\left( \frac{d \pi}{4} \right) \\
			|\mathrm{S}_0(Q^{(d)})| = 2^{d-2} - 2^{(d-2)/2} \cos\left( \frac{d \pi}{4} \right)&, \quad
			|\mathrm{S}_1(Q^{(d)})| = 2^{d-2} + 2^{(d-2)/2} \sin\left( \frac{d \pi}{4} \right)
		\end{align*}
		and therefore that the number of non-singular vectors is:
		\[
			|\mathrm{NS}(Q^{(d)})| = |\mathrm{NS}_0(Q^{(d+1)})| = 2^{d-1} + 2^{(d-1)/2} \cos\left(\frac{(d+1)\pi}{4}\right).
		\]
		Recall that we are assuming that $d$ is even, so $(d+1)$ is odd and $|\mathrm{NS}(Q^{(d)})| \neq 2^{d-1}$. By analysing the sign of $\cos((d+1)\pi/4)$, we obtain that there are more non-singular than singular vectors if $d \mod 8 \in \{0, 6\}$.
		
		Now assume $\pi^{(d)} = \sigma^{(d)}$. We have the following base cases for the recurrence (see the \nameref{sec:appendix} for computations): $|\mathrm{NS}_0(Q^{(8)})| = 56$, $|\mathrm{NS}_1(Q^{(8)})| = 64$, $|\mathrm{S}_0(Q^{(8)})| = 72$ and $|\mathrm{S}_1(Q^{(8)})| = 64$. We conclude that:
		\begin{align*}
			|\mathrm{NS}_0(Q^{(d)})| = 2^{d-2} - 2^{(d-2)/2} \cos\left( \frac{d \pi}{4} \right)&, \quad 
			|\mathrm{NS}_1(Q^{(d)})| = 2^{d-2} + 2^{(d-2)/2} \sin\left( \frac{d \pi}{4} \right) \\
			|\mathrm{S}_0(Q^{(d)})| = 2^{d-2} + 2^{(d-2)/2} \cos\left( \frac{d \pi}{4} \right)&, \quad
			|\mathrm{S}_1(Q^{(d)})| = 2^{d-2} - 2^{(d-2)/2} \sin\left( \frac{d \pi}{4} \right),
		\end{align*}
		so the number of non-singular vectors is:
		\[
			|\mathrm{NS}(Q^{(d)})| = |\mathrm{NS}_0(Q^{(d+1)})| = 2^{d-1} - 2^{(d-1)/2} \cos\left(\frac{(d+1)\pi}{4}\right).
		\]
		A similar analysis as in the previous case shows that there are more non-singular than singular vectors for if $d \mod 8 \in \{2, 4\}$.
	\end{proof}
	
	Now, we need to prove that some specific elements of $H_1(M^{(d)} \setminus \Sigma^{(d)})$ belong to $\{ e_\alpha  \}^{\Omega}_{\alpha \in \A}$.
	
	\begin{lem}\label{lem:alt_sum_1_1}
		If $\beta = 2 \mod{4}$, then $v = \sum_{\alpha = 1}^{\beta} (-1)^\alpha e_\alpha \in \{e_\alpha\}_{\alpha \in \A}^{\Omega^{(d)}}$.
	\end{lem}
	\begin{proof}
		First assume that $\pi^{(d)} = \tau^{(d)}$. We have that:
		\[
			1 = \langle e_2, -e_1 \rangle = \langle -e_1 + e_2, -e_4 \rangle = \langle -e_1 + e_2 + e_4, e_6 \rangle = \langle -e_1 + e_2 + e_4 + e_6, -e_3 \rangle,
		\]
		so $v_0 = -e_1 + e_2 - e_3 + e_4 + e_6 \in \{ e_\alpha \}_{\alpha \in \A}^{\Omega^{(d)}}$. If $\beta = 6$, we conclude since $v = v_0 - e_5$ and $1 = \langle v_0, -e_5\rangle$. We continue inductively: if $\beta = 6 + 4k$ with $k \geq 1$, we have that
		\begin{align*}
			1 &= \langle v_{k-1}, e_{3+4k} \rangle = \langle v_{k-1} -e_{3+4k}, -e_4\rangle = \langle v_{k-1} -e_{3+4k} - e_4, -e_{4+4k}\rangle \\
			&= \langle v_{k-1} -e_{3+4k} - e_4 + e_{4+4k}, -e_3\rangle = \langle v_{k-1} -e_{3+4k} - e_4 + e_{4+4k} + e_3, e_{5+4k}\rangle \\
			&= \langle v_{k-1} -e_{3+4k} - e_4 + e_{4+4k} + e_3 -e_{5+4k}, -e_4\rangle \\
			&= \langle v_{k-1} -e_{3+4k} - e_4 + e_{4+4k} + e_3 -e_{5+4k} + e_4, e_{6+4k}\rangle \\
			&= \langle v_{k-1} -e_{3+4k} - e_4 + e_{4+4k} + e_3 -e_{5+4k} + e_4 + e_{6+4k}, -e_3\rangle
		\end{align*}
		and we obtain that $v_k = v_{k-1} - e_{3+4k} + e_{4+4k} - e_{5+4k} + e_{6+4k} \in \{e_\alpha\}_{\alpha \in \A}^{\Omega^{(d)}}$. Since $v= v_k - e_5$ and $\langle v_k, -e_5\rangle = 1$, we obtain the desired result.

		Now assume that $\pi^{(d)} = \sigma^{(d)}$. Observe that:
		\begin{align*}
			1 &= \langle e_2, -e_1 \rangle = \langle -e_1 + e_2, -e_4 \rangle = \langle -e_1 + e_2 + e_4, e_8 \rangle = \langle -e_1 + e_2 + e_4 + e_8, -e_3 \rangle \\
			&= \langle -e_1 + e_2 -e_3 + e_4 + e_8, -e_5 \rangle = \langle -e_1 + e_2 -e_3 + e_4 -e_5 + e_8, -e_8 \rangle \\
			&= \langle -e_1 + e_2 -e_3 + e_4 -e_5, -e_6 \rangle,
		\end{align*}
		so $-e_1 + e_2 + e_4 + e_8, \sum_{\alpha = 1}^6 (-1)^\alpha e_\alpha \in  \{e_\alpha\}_{\alpha \in \A}^{\Omega^{(d)}}$. Now we will assume that $\beta \geq 10$.
	
		We have that:
		\begin{align*}
			1 &= \langle -e_1 + e_2 + e_4 + e_8, - e_5 \rangle = \langle -e_1 + e_2 + e_4 - e_5 + e_8, e_9 \rangle\\ 
			&= \langle -e_1 + e_2 + e_4 - e_5 + e_8 - e_9, -e_6 \rangle \\
			&= \langle -e_1 + e_2 + e_4 - e_5 + e_6 + e_8 - e_9, e_{10} \rangle \\
			&= \langle -e_1 + e_2 + e_4 - e_5 + e_6 + e_8 - e_9 + e_{10}, -e_3 \rangle 
		\end{align*}
		so $v_0 =  -e_1 + e_2 - e_3 + e_4 - e_5 + e_6 + e_8 - e_9 + e_{10} \in \{ e_\alpha \}_{\alpha \in \A}^{\Omega^{(d)}}$. If $\beta = 10$, we conclude since $v = v_0 - e_7$ and $1 = \langle v_0, -e_7\rangle$. We continue inductively: if $\beta = 10 + 4k$ with $k \geq 1$, we have that
		\begin{align*}
			1 &= \langle v_{k-1}, e_{7+4k} \rangle = \langle v_{k-1} -e_{7+4k}, -e_6\rangle = \langle v_{k-1} -e_{7+4k} - e_6, -e_{8+4k}\rangle \\
			&= \langle v_{k-1} -e_{7+4k} - e_6 + e_{8+4k}, -e_3\rangle = \langle v_{k-1} -e_{7+4k} - e_6 + e_{8+4k} + e_3, e_{9+4k}\rangle \\
			&= \langle v_{k-1} -e_{7+4k} - e_6 + e_{8+4k} + e_3 -e_{9+4k}, -e_6\rangle \\
			&= \langle v_{k-1} -e_{7+4k} - e_6 + e_{8+4k} + e_3 -e_{9+4k} + e_6, e_{10+4k}\rangle \\
			&= \langle v_{k-1} -e_{7+4k} - e_6 + e_{8+4k} + e_3 -e_{9+4k} + e_6 + e_{10+4k}, -e_3\rangle
		\end{align*}
		and we obtain that $v_k = v_{k-1} - e_{7+4k} + e_{8+4k} - e_{9+4k} + e_{10+4k} \in \{e_\alpha\}_{\alpha \in \A}^{\Omega^{(d)}}$. Since $v= v_k - e_7$ and $\langle v_k, -e_7\rangle = 1$, we obtain the desired result.
	\end{proof}
	
	For the rest of this section, we will fix $d$ until explicitly stated and write $M$ for $M^{(d)}$, $\Sigma$ for $\Sigma^{(d)}$, etc. Recall that $Q$ is the usual quadratic form on $H_1(M \setminus \Sigma; \Z / 2\Z)$ and that:
	\[
		Q(u) = \sum_{\alpha < \beta} u_\alpha \bOmega_{\alpha \beta} u_\beta + \sum_{\alpha \in \A} u_\alpha.
	\]
	We also denote by $\langle \cdot, \cdot\rangle$ the bilinear form on $H_1(M \setminus \Sigma; \Z / 2\Z)$ induced by $\bOmega$.
	
	\begin{defn}
		For $v \in H_1(M \setminus \Sigma; \Z / 2\Z)$ such that $Q(v) = 1$, we define the \emph{orthogonal transvection} along $v$ as:
		\[
			\bT_v(u) = u + \langle u, v \rangle v.
		\]
	\end{defn}
	
	Observe that $\bT_v^2 = \Id$ for each $v$ since $\bT_v^2(u) = \bT_v(u + \langle u, v \rangle v) = u + \langle u, v \rangle v + \langle u, v \rangle v = u$. Moreover, it is not hard to see that if $v \in H_1(M \setminus \Sigma)$ then $\bT_{\text{\xoverline{v}}}$ belongs to $O(Q)$ if and only if $Q(v) = 1$, so the set of orthogonal transvections coincides with the modulus-two reduction of the set of the symplectic transvections whose reductions preserve $Q$.
		
	\begin{lem}\label{lem:generate_phi-closed}
		Let $v, w \in H_1(M_\pi \setminus \Sigma_\pi; \Z / 2\Z)$ such that $Q(v) = Q(w) = Q(v+w) = 1$. Then, $\bT_v \bT_w \bT_v = \bT_w \bT_v \bT_w = \bT_{v+w}$.
	\end{lem}
	
	\begin{proof}
		Since $Q(v+w) = Q(v) + Q(w) + \langle v, w \rangle$, we obtain that $\langle v, w\rangle = \langle w, v\rangle = 1$. The proof then follows from \Cref{lem:generate_dehn_twists}, since we can find $\tilde{v}, \tilde{w} \in H_1(M_\pi \setminus \Sigma_\pi)$ such that $\langle \tilde{v}, \tilde{w}\rangle = 1$ and $v = \tilde{v} \mod 2$, $w = \tilde{w} \mod 2$.
	\end{proof}
	
	We denote the set of non-singular vectors by $\NS(Q) \subseteq H_1(M \setminus \Sigma; \Z / 2\Z)$ and the set of singular vectors by $\mathrm{S}(Q) \subseteq H_1(M \setminus \Sigma; \Z / 2\Z)$.
	
	\begin{defn}
		For $V \subseteq \NS(Q)$, let $\xoverline{G}_V \subseteq O(Q)$ be the group spanned by the orthogonal transvections $\{\bT_v\}_{v \in V}$. We say that a set $X \subseteq \NS(Q)$ is \emph{$Q$-closed} if $Q(v + w) = 1$ implies $v+w \in X$ for every $v, w \in X$. We denote by $V^Q$ the smallest $Q$-closed set containing $V$. By the previous lemma, $\{\bT_v\}_{v \in V^Q} \subseteq \xoverline{G}_V$.
	\end{defn}
	
	Observe that this definition is analogous to \Cref{def:omega_closed}. Indeed, the modulus-two reduction of a non-singular $\Omega$-closed set is a $Q$-closed set.
	
	It is known from the theory of classical groups that orthogonal transvections generate $O(Q)$ if the dimension of the vector space is at least $6$ and the alternate form is nondegenerate \cite[Theorem 14.16]{G:groups}. This holds for every even $d$. Since $\bT_\alpha = \bT_{\text{\be}_\alpha}$ belongs to $\bRV(\pi)$ for every $\alpha \in \A$, it is enough to prove that $\{\be_\alpha\}_{\alpha \in \A}^Q = \NS(Q)$ to obtain that $\bRV(\pi) = O(Q)$. We will need one extra auxiliary lemma:
	
	\begin{lem} \label{lem:kernel_omega}
		If $d\mod{4} = 2$, then $v = \sum_{\alpha \in \A} \be_\alpha, v_\beta = \sum_{\alpha \neq \beta} \be_\alpha \in \{\be_\alpha\}_{\alpha \in \A}^Q$ for each $\beta \in \A$.
	\end{lem}
	
	\begin{proof}
		We have that $\sum_{\alpha \in \A} (-1)^\alpha e_\alpha \in \{e_\alpha\}_{\alpha \in \A}^{\Omega}$ by \Cref{lem:alt_sum_1_1}. By taking the modulus-two reduction, this implies that $v = \sum_{\alpha \in \A} \be_\alpha \in \{\be_\alpha\}_{\alpha \in \A}^{Q}$.

		Now, for any $\beta \in \A$ we have that:
		\[
			Q(v_\beta) = Q(v + e_\beta) = 1 + 1 + \sum_{\alpha \in \A} \langle e_\alpha, e_\beta \rangle = 1,
		\]
		so $v_\beta \in \{\be_\alpha\}_{\alpha \in \A}^Q$.
	\end{proof}
	Observe that $\sum_{\alpha \in \A} \be_\alpha \in \NS(Q^{(d)})$ if and only if $d \mod 4 \in \{1,2\}$. We can now prove the final lemma of this section:
	\begin{lem} \label{lem:allorthogonaltransvections}
		For each $d$, we have $\{\be_\alpha\}_{\alpha \in \A}^Q = \NS(Q) \setminus \ker \bOmega$.
	\end{lem}
	
	\begin{proof}
		The proof is by induction. The base cases, $d = 6$ for $\tau^{(d)}$ and $d = 8$ for $\sigma^{(d)}$, can be done computationally (see the \nameref{sec:appendix} for computations). The inductive step is exactly the same for both families of irreducible permutations.
		
		We have that
		\[
			\bOmega^{(d)} =
			\begin{pmatrix}
				\bOmega^{(d-1)} & \be_\sharp^\tr \\
				\be_\sharp & 0
			\end{pmatrix}
		\]
		where $\be_\sharp = \sum_{\alpha < d} \be_\alpha$.
		Observe that, for any $u \in H_1(M^{(d)} \setminus \Sigma^{(d)}; \Z / 2\Z)$,
		\[
			Q^{(d)}(u) = \sum_{\alpha < \beta < d} u_\alpha \bOmega^{(d-1)}_{\alpha \beta} u_\beta + u_d\sum_{\alpha < d} u_\alpha + \sum_{\alpha \in \A} u_\alpha = \sum_{\alpha < \beta < d} u_\alpha \bOmega^{(d-1)}_{\alpha \beta} u_\beta + (1+u_d) \sum_{\alpha < d}u_\alpha + u_d.
		\]
		Let $v \in H_1(M^{(d)} \setminus \Sigma^{(d)}; \Z / 2\Z)$ be such that $v_\alpha = u_\alpha$ for $\alpha < d$ and $v_d = 0$. We define $p_{d-1}(u) \in H_1(M^{(d-1)} \setminus \Sigma^{(d-1)}; \Z / 2\Z)$ to be the projection of $u$ by removing the $d$-th coordinate. Then:
		\[
			Q^{(d)}(v) = \sum_{\alpha < \beta < d} u_\alpha \bOmega^{(d-1)}_{\alpha \beta} u_\beta + \sum_{\alpha < d} u_\alpha = Q^{(d-1)}(p_{d-1}(u)), \text{ so } Q^{(d)}(u) = Q^{(d)}(v) + u_d\sum_{\alpha \in \A} u_\alpha.
		\]
		
		Assume now that $u \in \mathrm{NS}(Q^{(d)}) \setminus \ker \bOmega$. We consider several cases:
		
		If $u_d = 0$, we have that $Q^{(d-1)}(p_{d-1}(u)) = 1$. If $d$ is odd, then $\bOmega^{(d-1)}$ is invertible, so $p_{d-1}(u) \notin \ker \bOmega^{(d-1)}$ and we obtain that $u \in \{\be_\alpha\}_{\alpha \in \A}^{Q^{(d)}}$. If, on the contrary, $d$ is even, then $\ker \bOmega^{(d-1)} = \{0, \be_\sharp\}$. If $u = (\be_\sharp\ \ 0)$, then $d = 2 \mod{4}$ and can use the previous lemma to see that $u \in \{\be_\alpha\}_{\alpha \in \A}^{Q^{(d)}}$. If $u \neq (\be_\sharp\ \ 0)$, then $p_{d-1}(u) \notin \ker \bOmega^{(d-1)}$, so we can use the induction hypothesis.
		
		If $u_d = 1$, we have that $u = v + \be_d$. If $\sum_{\alpha \in \A} u_\alpha = 0$, or, equivalently, if $\sum_{\alpha < d} u_\alpha = 1$, then $Q^{(d)}(v) = Q^{(d)}(u) = 1$. If $d$ is odd, then $v \in \{\be_\alpha\}_{\alpha \in \A}^{Q^{(d)}}$ by induction hypothesis. If $d$ is even and $p_{d-1}(u) = \be_\sharp$, then $d = 2 \mod{4}$ and we can use the previous lemma to obtain that $v \in \{\be_\alpha\}_{\alpha \in \A}^{Q^{(d)}}$. Otherwise, we can use the induction hypothesis to obtain the same result. In any case, $u = v + \be_d \in \{\be_\alpha\}_{\alpha \in \A}^{Q^{(d)}}$ by $Q^{(d)}$-closedness. Finally, assume that $\sum_{\alpha < d} u_\alpha = 0$, so $Q^{(d)}(v) = 0$. Thus,
		\[
				Q^{(d)}(v) = Q^{(d)}(u + \be_d) = Q^{(d)}(u) + Q^{(d)}(\be_d) + \langle u, \be_d\rangle = 0.
		\]
		Since $u \notin \ker \bOmega$, there exists $\beta \in \A$ with $\langle u, \be_\beta\rangle = 1$ so, $Q^{(d)}(u + \be_\beta) = \langle u, \be_\beta\rangle = 1$. We know that $\beta < d$, since $Q^{(d)}(v) = 0$. Moreover, from $\bOmega_{\beta d} = 1$, we can see that:
		\[
			Q^{(d)}(u + \be_\beta + \be_d) = Q^{(d)}(u + \be_\beta) + Q^{(d)}(\be_d) + \langle u + \be_\beta, \be_d\rangle = 1.
		\]
		If $d$ is odd, then $\bOmega^{(d-1)}$ is invertible, so $p_{d-1}(u + \be_\beta + \be_d) \notin \ker \bOmega^{(d-1)}$ and we conclude by induction hypothesis and $Q^{(d)}$-closedness since $(u + \be_\beta + \be_d) + (e_\beta + \be_d) = u$. If, on the contrary, $d$ is even, we conclude in the same way if $u + \be_\beta + \be_d \neq (\be_\sharp\ \ 0)$. Otherwise, $d = 2 \mod 4$ and $u = \sum_{\alpha \neq \beta} \be_\alpha$, so we conclude by the previous lemma.
	\end{proof}
	
	We obtain the following proposition which completes the proof of Theorem \ref{thm:minimal}:
	\begin{prop} \label{lem:surjectivity}
		For even $d$, the orthogonal transvections $\{\bT_\alpha\}_{\alpha \in \A}$ generate the orthogonal group $O(Q^{(d)})$. In particular, it is contained in $\bRV(\pi^{(d)})$.
	\end{prop}
	\begin{proof}
		By the previous lemma, we have that $\{e_\alpha\}_{\alpha \in \A}^{Q^{(d)}} = \NS(Q^{(d)})$. Since the dimension of $H_1(M^{(d)} \setminus \Sigma^{(d)}; \Z / 2\Z)$ is at least $6$, the orthogonal transvections generate $O(Q^{(d)})$. We conclude by \Cref{lem:generate_phi-closed}.
	\end{proof}
	
	\section{Non-hyperelliptic connected components of \texorpdfstring{$\H(g-1,g-1)$}{H(g-1,g-1)}} \label{sec:odd_d}
	
	As was done by Avila, Matheus and Yoccoz for the hyperelliptic connected components \cite{AMY:hyperelliptic}, we can give an explicit description of the Rauzy--Veech groups for the remaining connected components of $\H(g-1,g-1)$. We will continue using the representatives found in the previous section. We will prove the following theorem:
	\smallbreak
	\begin{thm} \label{thm:odd_d}
		For any $g \geq 3$, the Rauzy--Veech group of a non-hyperelliptic connected component of $\H(g - 1, g - 1)$ is equal to the preimage of the orthogonal group $O(Q^{(d)})$ by the modulus-two reduction $\Sp(\Omega^{(d)}, \Z) \cap \SL(H_1(M^{(d)} \setminus \Sigma^{(d)})) \to \Sp(\bOmega^{(d)}, \Z / 2\Z)$, where $d = 2g+1$. If $g$ is odd, it is isomorphic to $\RV(\pi^{(d-1)}) \ltimes \Z^{d-1}$. If $g$ is even, it is isomorphic to a finite-index subgroup of $\Sp(\Omega^{(d-1)}, \Z) \ltimes \Z^{d-1}$. Moreover, such groups are generated by $\{T_\alpha\}_{\alpha \in \A}$.
	\end{thm}
	\smallbreak
	
	Assume for the rest of this section that $d \geq 7$ is odd. Let $G_d$ be the preimage in the statement of the theorem. We have that $\RV(\pi^{(d)}) \subseteq G_d$ by \Cref{lem:preserve_orthogonal}. We will prove that every element of $G_d$ belongs to $\RV(\pi^{(d)})$.
	
	Recall that $e_\sharp = (1,-1,1,\dotsc,-1,1)$ and that $\ker \Omega^{(d)}$ is generated by $e_\sharp$. Consider the $\Z$-submodule $V$ of $H_1(M^{(d)} \setminus \Sigma^{(d)})$ spanned by $\{ e_\alpha \}_{\alpha < d}$. Observe that $\Omega^{(d)}|_V$ is nondegenerate and that $V \oplus \ker \Omega^{(d)} = H_1(M^{(d)} \setminus \Sigma^{(d)})$. We have the following decomposition:
	
	\begin{lem}
		For any $S \in G_d$, we have that $S$ acts as the identity on $\ker \Omega^{(d)}$. Moreover, if $S^0 \colon V \to \ker \Omega^{(d)}$ and $S^1 \colon V \to V$ are the unique linear maps satisfying $S|_V = S^0 + S^1$, then $S^1 \in \Sp(\Omega^{(d)}|_V, \Z)$.
	\end{lem}
	\begin{proof}
		The fact that $S^1 \in \Sp(\Omega^{(d)}|_V, \Z)$ follows from a straightforward computation:
		\[
			\langle S^1(u), S^1(v) \rangle = \langle S^1(u) + S^0(u), S^1(v) + S^0(v) \rangle = \langle S(u), S(v) \rangle = \langle u, v \rangle
		\]
		for any $u, v \in V$. Since $\Omega^{(d)}|_V$ is nondegenerate we obtain that $\det S|_V = 1$. We have that $S$ preserves $\ker \Omega^{(d)}$. In other words, $e_\sharp$ is an eigenvector of $S$. Therefore, we obtain that $1 = \det S = \det S|_V \det S|_{\ker \Omega^{(d)}}$. Thus, $\det S|_{\ker \Omega^{(d)}} = 1$, so $S|_{\ker \Omega^{(d)}} = \Id_{\ker \Omega^{(d)}}$.
	\end{proof}
		
	For any $S, T \in G_d$, we have that $(TS)^0 = T^0S^1 + S^0$ and that $(TS)^1 = T^1S^1$. In particular, $\{ S^1 \ \mid\ S \in G_d \}$ is a subgroup of $\Sp(\Omega^{(d)}|_V, \Z)$.
	
	\begin{lem}
		The group $O(Q^{(d)})$ is isomorphic to $\Sp(\bOmega^{(d-1)}, \Z / 2\Z)$ if $d = 1 \mod{4}$ and to $O(Q^{(d-1)}) \ltimes (\Z / 2\Z)^{d-1}$ if $d = 3 \mod{4}$.
	\end{lem}
	\begin{proof}
		If $d = 1 \mod {4}$, then $Q^{(d)}(\be_\sharp) = 1$. Therefore, $Q^{(d)}$ is \emph{regular}, that is, the only element of $\ker \Omega^{(d)} \cap \NS(Q^{(d)})$ is $0$. In this case, the map $S \mapsto S^1$ from $O(Q^{(d)})$ to $\Sp(\bOmega^{(d)}|_{\text{\bV}}, \Z / 2\Z)$ is an isomorphism \cite[Theorem 14.1, Theorem 14.2]{G:groups}.
	
		If $d = 3 \mod{4}$, then $Q^{(d)}(\be_\sharp) = 0$, so $Q^{(d)}$ is not regular. The restriction of $Q^{(d)}$ to $\bV$ is a quadratic form for the nondegenerate symplectic form $\bOmega^{(d)}|_{\text{\bV}}$. The natural projection $p_{d-1} \colon \bV \to H_1(M^{(d-1)} \setminus \Sigma^{(d-1)}; \Z / 2\Z)$ to the first $d-1$ coordinates is an isomorphism between $Q^{(d)}|_{\text{\bV}}$ and $Q^{(d-1)}$. In particular, the Arf invariants of $Q^{(d)}|_{\text{\bV}}$ and $Q^{(d-1)}$ coincide. Observe that $S^1 \in O(Q^{(d)}|_{\text{\bV}})$ for any $S \in O(Q^{(d)})$, since
		\begin{align*}
			Q^{(d)}(S^1(u)) &= Q^{(d)}(S^1(u) + S^0(u) + S^0(u)) = Q^{(d)}(S(u) + S^0(u)) \\
			&= Q^{(d)}(S(u)) + Q^{(d)}(S^0(u)) + \langle S(u), S^0(u) \rangle = Q^{(d)}(u)
		\end{align*}
		for any $u \in \bV$, so we obtain that $\{S^1 \ \mathbin{|}\ S \in O(Q^{(d)})\} \subseteq O(Q^{(d)}|_{\text{\bV}})$. By \Cref{lem:allorthogonaltransvections}, we have that $\{p_{d-1}(\be_\alpha)\}_{\alpha < d}^{Q^{(d-1)}} = \NS(Q^{(d-1)})$ and therefore that $\{\be_\alpha\}_{\alpha < d}^{Q^{(d)}} = \bV \cap \NS(Q^{(d)})$, showing that $\{S^1 \ \mathbin{|}\ S \in O(Q^{(d)})\} = O(Q^{(d)}|_{\text{\bV}})$, which is isomorphic to $O(Q^{(d-1)})$.
		
		Furthermore, for any $S^1 \in O(Q|_{\text{\bV}})$ we can choose any linear map $S^0 \colon \bV \to \ker \bOmega^{(d)}$ and define $S$ as the identity on $\ker \bOmega^{(d)}$ and as $S^0 + S^1$ on $\bV$. We have that $S \in O(Q^{(d)})$ since:
		\[
			Q(S(u)) = Q(S^0(u) + S^1(u)) = Q(S^0(u)) + Q(S^1(u)) + \langle S^0(u), S^1(u) \rangle = Q(S^1(u)) = Q(u)
		\]
		for every $u \in V$.
		
		Finally, for every $S^0, T^0 \colon \bV \to \ker \bOmega^{(d)}$ there exist unique elements $v, w$ of $\bV$ such that $S^0(u) = \langle u, v \rangle \be_\sharp$ and $T^0(u) = \langle u, v \rangle \be_\sharp$ for every $u \in \bV$. We identify $S^0$ with $v$ and $T^0$ with $w$. From the equality $(TS)^0 = T^0 S^1 + S^0$, we obtain that
		\[
			(TS)^0(u) = \langle S^1(u), w\rangle \be_\sharp + \langle u, v\rangle \be_\sharp = \langle u, (S^1)^\tr(w) + v \rangle \be_\sharp,
		\]
		where $(S^1)^\tr$ is the transpose of $S^1$ for the non-degenerate symplectic form $\bOmega^{(d)}|_{\text{\bV}}$. This shows that $O(Q^{(d)})$ is isomorphic to $O(Q^{(d)}|_{\text{\bV}}) \ltimes \bV$ and to $O(Q^{(d-1)}) \ltimes (\Z / 2\Z)^{d-1}$.		
	\end{proof}
	
	The following two propositions complete the proof of \Cref{thm:odd_d}.
		
	\begin{prop} \label{prop:odd_d_even_g}
		For any even $g \geq 4$, the Rauzy--Veech group of $\H(g - 1, g-1)$ is equal to $G_d$, where $d = 2g + 1$. It is isomorphic to a subgroup of $\Sp(\Omega^{(d-1)}, \Z) \ltimes \Z^{d-1}$ of index $2^g$ and it is generated by $\{T_\alpha\}_{\alpha \in \A}$.
	\end{prop}
	
	\begin{proof}
		Let $S \in G_d$. We will prove that $S \in \RV(\pi^{(d)})$. First, we can find $T \in \RV(\pi^{(d)})$ such that $T^1 = S^1$. This can be done since $p_{d-1} \colon V \to H_1(M^{(d-1)} \setminus \Sigma^{(d-1)})$ is an isomorphism conjugating the action of $\{ T^1 \ |\ T \in \RV(\pi^{(d)}) \}$ and $\Sp(\Omega^{(d-1)}, \Z)$. Indeed, by \Cref{lem:transvections_rauzy_class} and the results of the previous two sections, the action of the subgroup of $\{ T^1 \ |\ T \in \RV(\pi^{(d)}) \}$ generated by the maps $\{ T_\alpha^1 \ \mid\ \alpha < d \}$ is conjugated by $p_{d-1}$ to $\RV(\pi^{(d-1)}) \subseteq \Sp(\Omega^{(d-1)}, \Z)$. The group $\RV(\pi^{(d-1)})$ is maximal in $\Sp(\Omega^{(d-1)}, \Z)$ in the sense that if $\RV(\pi^{(d-1)}) \subseteq G \subseteq \Sp(\Omega^{(d-1)}, \Z)$ is a group, then it is either $\RV(\pi^{(d-1)})$ or $\Sp(\Omega^{(d-1)}, \Z)$ \cite[Theorem 3]{BGP:finiteindex} \footnote{This result uses the classification of finite simple groups.}. Since $\bT_d^1$ does not preserve $Q^{(d)}|_{\text{\bV}}$, we have that $(p_{d-1})_* T_d^1 \notin \RV(\pi^{(d-1)})$, showing that $\{ T^1 \ |\ T \in \RV(\pi^{(d)}) \} = \Sp(\Omega^{(d)}|_V, \Z)$. 
		
		Now observe that, for any $u, v \in V$,
		\begin{align*}
			T_{v - e_\sharp}^{-2}T_v^2(u) &= T_{v - e_\sharp}^{-2}(u + 2\langle u, v \rangle v) = u + 2\langle u, v \rangle v - 2\langle u + 2\langle u, v \rangle v, v + e_\sharp\rangle (v - e_\sharp) \\
			&= u + \langle 2u, v\rangle e_\sharp.
		\end{align*}
		
		We set $S_v = T_{v - e_\sharp}^{-2}T_v^2(u)$. Clearly $S_v^1 = \Id|_V$, so $(S_v S_w)^1 = \Id|_V$ and $(S_v S_w)^0 = S_{v+w}^0$ for any $v, w \in V$.
		
		We will now show that $S_v \in \RV(\pi^{(d)})$ for every $v \in V$. Indeed, we start by showing that $T_{v - e_\sharp}^2 \in \RV(\pi^{(d)})$ for every $\alpha < d$. By \Cref{lem:alt_sum_1_1}, we have that $w = \sum_{\alpha = 1}^{d - 3} (-1)^\alpha e_\alpha$ belongs to $\{e_\alpha\}_{\alpha \in \A}^{\Omega^{(d)}}$. We will consider several cases.
		
		We can assume that $\pi^{(d)} = \tau^{(d)}$, since $\tau^{(d)}$ and $\sigma^{(d)}$ represent the same connected component. If $\alpha \leq d-3$, then $1 = \langle e_\alpha, e_{d-2} \rangle = \langle e_\alpha + e_{d-2}, w \rangle$, so $e_\alpha - w + e_{d-2} \in \{e_\alpha\}_{\alpha \in \A}^{\Omega^{(d)}}$. We can use \Cref{cor:generate_dehn_twists_nosign} with $v_1' = -e_{d-1} + e_d$, $v_2' = e_\alpha - w + e_{d-2}$, $v_3' = w$ and $v_4' = e_d - (-1)^\beta e_\beta$, where $\beta = 4$ if $\alpha = 1$ and $\beta = 1$ if $\alpha > 1$. We obtain that $T_{v-e_\sharp}^2 \in \RV(\pi^{(d)})$ if $\alpha \leq d - 3$. Now assume that $\alpha = d-2$. We have that
		\[1 = \langle e_{d-2}, e_{d-1}\rangle = \langle e_{d-2} - e_{d-1}, e_{d-2}\rangle = \langle 2e_{d-2} - e_{d-1}, e_d\rangle,
		\]
		so $2e_{d-2} - e_{d-1} + e_d \in \{e_\alpha\}_{\alpha \in \A}^{\Omega^{(d)}}$. We can use \Cref{cor:generate_dehn_twists_nosign} by choosing $v_1' = -w$, $v_2' = 2e_{d-2} - e_{d-1} + e_d$, $v_3' = e_d$ and $v_4' = e_d + e_1$, so we obtain that $T_{e_{d-2} - e_\sharp}^2 \in \RV(\pi^{(d)})$. Finally, assume that $\alpha = e_{d-1}$. We can use \Cref{cor:generate_dehn_twists_nosign} with $v_1' = -w$, $v_2' = e_{d-2} + e_d$, $v_3' = e_d$ and $v_4' = e_d + e_1$. We obtain that $T_{e_\alpha - e_\sharp}^2 \in \RV(\pi^{(d)})$ for every $\alpha < d$, so $S_{e_\alpha} \in \RV(\pi^{(d)})$ for every $\alpha < d$.
		
		By writing $v = \sum_{\alpha < d} n_\alpha e_\alpha$, we get that $S_v = S_{e_\alpha}^{n_\alpha} \dotsb S_{e_{d-1}}^{n_{d-1}} \in \RV(\pi^{(d)})$ for every $v \in V$.
		
		Finally, let $v, w \in V$ such that $S^0 = \langle \cdot, v\rangle e_\sharp$ and $T^0 = \langle \cdot, w\rangle e_\sharp$. Since $S^1 = T^1$, we have that $\xoverline{S}^0 = \bT^0$ since the map $O(Q^{(d)}) \to \Sp(\bOmega^{(d)}, \Z / 2\Z)$ is injective. We obtain that $v = w \mod {2}$, so there exists $u \in V$ such that $2u = v - w$. Observe that $(TS_u)^1 = T^1 = S^1$ and that $(TS_u)^0 = T^0 + S_u^0 = S^0$, so $S = TS_u \in \RV(\pi^{(d)})$.
		
		The group $G_d$ is isomorphic to a subgroup of $\Sp(\Omega^{(d)}|_V, \Z) \ltimes V$ that we will soon describe. Let $S^1 \in \Sp(\bOmega^{(d)}|_{\text{\bV}}, \Z / 2\Z)$. There exists a unique linear map $S^0 \colon \bV \to \ker \bOmega^{(d)}$ such that the map $S$ defined as the identity on $\ker \bOmega^{(d)}$ and as $S^0 + S^1$ on $\bV$ belongs to $O(Q^{(d)})$. We can characterize $S^0$ as follows: $S^0(u)$ is the unique element of $\ker \Omega^{(d)}$ satisfying $Q^{(d)}(S^0(u)) = Q^{(d)}(S^1(u)) + Q^{(d)}(u)$ \cite[Theorem 14.1]{G:groups}. We obtain that:
		\[
			S^0(u) = \begin{cases}
				0 & Q^{(d)}(S^1(u)) = Q^{(d)}(u) \\
				\be_\sharp & Q^{(d)}(S^1(u)) \neq Q^{(d)}(u).
			\end{cases}
		\]
		If $S^1, T^1 \in \Sp(\bOmega^{(d)}|_{\text{\bV}}, \Z / 2\Z)$, observe that $Q^{(d)}(S^1(u)) = Q^{(d)}(T^1(u))$ for every $u \in \bV$ if and only if they belong to the same coset of $\Sp(\bOmega^{(d)}|_{\text{\bV}}, \Z / 2\Z) / O(Q^{(d)}|_{\text{\bV}})$. Let $v_{S^1} \in \bV$ be the unique element of $\bV$ such that $S^0 = \langle \cdot, v_{S^1} \rangle \be_\sharp$. The image of the map $S^1 \mapsto v_{S^1}$ consists of $2^{g-1}(2^g \pm 1)$ elements, depending on the Arf invariant of $O(Q^{(d)}|_{\text{\bV}})$.
		
		We can now describe $G_d$ up to isomorphism as follows: for each $S^1 \in \Sp(\Omega^{(d)}|_V, \Z)$, we define $V_{S^1} = \{ w \in V \ \mid\ \xoverline{w} = v_{\bar{S}^1} \}$. Then, $G_d$ is isomorphic to $\bigcup_{S^1 \in \Sp(\Omega^{(d)}|_V, \Z)} \{S^1\} \times V_{S^1}$, regarded as subgroup of $\Sp(\Omega^{(d)}|_V, \Z) \ltimes V$. It has index $2^g$ in $\Sp(\Omega^{(d)}|_V, \Z) \ltimes V$ since a coset $C$ is determined by the unique vector $v \in \bV$ such that $(\Id|_V, w) \in C$ for every $w \in V$ with $\xoverline{w} = v$.
	\end{proof}
		
	\begin{prop} \label{prop:odd_d_odd_g}
		For any odd $g \geq 3$, the Rauzy--Veech group of $\H(g - 1, g-1)$ is equal to $G_d$, where $d = 2g + 1$. It is isomorphic to $\RV(\pi^{(d-1)}) \ltimes \Z^{d-1}$ and it is generated by $\{T_\alpha\}_{\alpha \in \A}$.
	\end{prop}
	
	\begin{proof}
		The proof is very similar to that of \Cref{prop:odd_d_even_g}. Let $S \in G_d$. We will prove that $S \in \RV(\pi^{(d)})$. First, we can find $T \in \RV(\pi^{(d)})$ such that $T^1 = S^1$. This can be done since $p_{d-1} \colon V \to H_1(M^{(d-1)} \setminus \Sigma^{(d-1)})$ is an isomorphism conjugating the action of $\{ T^1 \ |\ T \in \RV(\pi^{(d)}) \}$ and $\RV(\pi^{(d-1)})$ by \Cref{lem:transvections_rauzy_class} and our previous results.
		
		Now observe that, for any $u, v \in V$,
		\[
			T_{v - e_\sharp}^{-1}T_v(u) = T_{v - e_\sharp}^{-1}(u + \langle u, v \rangle v) = u + \langle u, v \rangle v - \langle u + \langle u, v \rangle v, v + e_\sharp\rangle (v - e_\sharp) = u + \langle u, v\rangle e_\sharp.
		\]
		We set $S_v = T_{v - e_\sharp}^{-1}T_v(u)$. Clearly $S_v^1 = \Id|_V$, so $(S_v S_w)^1 = \Id|_V$ and $(S_v S_w)^0 = S_{v+w}^0$ for any $v, w \in V$.
		
		We will now show that $S_v \in \RV(\pi^{(d)})$ for every $v \in V$. Indeed, we start by showing that $e_\alpha - e_\sharp \in \{e_\alpha\}_{\alpha \in \A}^{\Omega^{(d)}}$ for every $\alpha < d$. Indeed, by \Cref{lem:alt_sum_1_1}, $w = \sum_{\alpha = 1}^{d - 1} (-1)^\alpha e_\alpha$ belongs to $\{e_\alpha\}_{\alpha \in \A}^{\Omega^{(d)}}$. Since $1 = \langle e_\alpha, e_d \rangle = \langle e_\alpha + e_d, w \rangle$, we obtain that $e_\alpha + e_d - w = e_\alpha - e_\sharp \in \{e_\alpha\}_{\alpha \in \A}^{\Omega^{(d)}}$. We conclude that $S_{e_\alpha} \in \RV(\pi^{(d)})$ for every $\alpha < d$. By writing $v = \sum_{\alpha < d} n_\alpha e_\alpha$, we get that $S_v = S_{e_\alpha}^{n_\alpha} \dotsb S_{e_{d-1}}^{n_{d-1}} \in \RV(\pi^{(d)})$ for every $v \in V$.
		
		Finally, let $v, w \in V$ such that $S^0 = S_v^0$ and $T^0 = S_w^0$. Observe that $(TS_{v-w})^1 = T^1 = S^1$ and that $(TS_{v-w})^0 = T^0 + S_{v-w}^0 = S^0$, so $S = TS_{v-w} \in \RV(\pi^{(d)})$.
	\end{proof}

	\section{Rauzy--Veech groups of general strata} \label{sec:reduction}
	
	We start by stating in a more precise manner what the Zariski-density means in the case of strata with more than one marked point. Let $\pi$ be an irreducible permutation and let $V$ be any complement of $\ker \Omega_\pi$ inside $H_1(M_\pi \setminus \Sigma_\pi)$. We have that $\Omega_\pi|_V$ is a nondegenerate symplectic form on $V$. For a map $S \in \RV(\pi)$, we define $S^1 \colon V \to V$ to be the projection of $S|_V$ on $V$. Then, we define the Rauzy--Veech group on $V$  as the set $\RV(\pi)|_V = \{ S^1 \ \mid\ S \in \RV(\pi) \}$, which is a subgroup of $\Sp(\Omega_\pi|_V, \Z)$. If $W$ is another complement of $\ker \Omega_\pi$ and $\varphi \colon V \to W$ is any symplectic isomorphism, one has that $\varphi$ conjugates the actions of $\Sp(\Omega_\pi|_V, \Z)$ and $\Sp(\Omega_\pi|_W, \Z)$ and the actions of $\RV(\pi)|_V$ and $\RV(\pi)|_W$. Therefore, we say that a Rauzy--Veech group is Zariski-dense if $\RV(\pi)|_V$ is Zariski-dense inside $\Sp(\Omega_\pi|_V, \R)$.
	
	Using the adjacency of strata, we will prove that $\RV(\pi)|_V$ contains the Rauzy--Veech group of a connected component of a minimal stratum. We will then analyse which strata are adjacent to each other, which will conclude the proof of \Cref{thm:general} since the Rauzy--Veech group of any connected component of a minimal stratum is Zariski-dense by the classification of connected components \cite{KZ:connected_components}, the work on the hyperelliptic case by Avila, Matheus and Yoccoz \cite{AMY:hyperelliptic} and \Cref{thm:minimal}.

	The following concepts were introduced by Avila and Viana \cite[Section 5]{AV:KZ_conjecture}:
	
	\begin{defn}
		Let $\pi$ be an irreducible permutation on an alphabet $\A$. Let $\alpha \in \A$ and let $\pi'$ be the permutation on $\A \setminus \{\alpha\}$ obtained by erasing the letter $\alpha$ from the top and bottom rows of $\pi$. If $\pi'$ is irreducible, we say that it is a \emph{simple reduction} of $\pi$.
	\end{defn}
	
	We also need a slightly stronger definition:
	
	\begin{defn}
		Let $\pi'$ be an irreducible permutation on an alphabet $\A'$ not containing $\alpha$. Let $\beta, \beta' \in \A'$ such that $(\alpha_{\mathrm{t}, 1}, \alpha_{\mathrm{b}, 1}) \neq (\beta, \beta')$. We define the permutation $\pi$ on the alphabet $\A' \cup \{\alpha\}$ by inserting $\alpha$ just before $\beta$ in the top row and just before $\beta'$ in the bottom row of $\pi'$. We say that $\pi$ is a \emph{simple extension} of $\pi'$. 
	\end{defn}
	
	This definition can be extended, by the same rule, to the Rauzy class of $\pi'$. We denote by $\mathcal{E}$ the \emph{extension map} defined in this way.

	If $\pi$ is a simple extension of $\pi'$, then $\pi$ is also irreducible \cite[Lemma 5.4]{AV:KZ_conjecture} and $\pi'$ is a simple reduction of $\pi$. Conversely, if $\pi'$ is a simple reduction of $\pi$ whose omitted letter is not the last on the top or bottom row of $\pi$, then $\pi$ is a simple extension of $\pi'$.
	
	Simple extensions allow us to find copies of simpler Rauzy--Veech groups inside more complex ones. Indeed, we start by recalling the definition of the extension map. Let $\gamma'$ be an arrow in the Rauzy class of $\pi'$ starting at $\pi'$. We define a path $\mathcal{E}_*(\gamma')$ in the Rauzy class of $\pi$ starting at $\pi$ as follows: (1) if $\gamma'$ is a top arrow and the letter $\alpha$ is added before the last letter on the bottom row of $\pi'$, then $\mathcal{E}_*(\gamma')$ is constructed by applying two top operations to $\pi$; (2) if $\gamma'$ is a bottom arrow and the letter $\alpha$ is added before the last letter on the top row of $\pi'$, then $\mathcal{E}_*(\gamma')$ is constructed by applying two bottom operations to $\pi$; (3) otherwise, $\mathcal{E}_*(\gamma')$ is constructed by applying one operation of the same type as $\gamma'$ to $\pi$. In every case, $\mathcal{E}_*(\gamma')$ starts at the image by $\mathcal{E}$ of the start of $\gamma'$ and ends at the image by $\mathcal{E}$ of the end of $\gamma'$. This definition can be extended to any walk on the Rauzy class of $\pi'$ by concatenation. We refer the reader to the work of Avila and Viana for more details \cite[Section 5.2]{AV:KZ_conjecture}. 
	
	We restate a lemma \cite[Lemma 5.6]{AV:KZ_conjecture} from their work to match our context:
	\begin{lem}
		Let $\pi$ be a genus-preserving simple extension of $\pi'$ and let $V'$ be a complement of $\ker \Omega_{\pi'}$. Then, the injection $\iota \colon H_1(M_{\pi'} \setminus \Sigma_{\pi'}) \to H_1(M_\pi \setminus \Sigma_{\pi})$ restricts to $V'$ as a symplectic isomorphism between $V'$ and $V = \iota(V')$ such that $V$ is a complement of $\ker \Omega_\pi$. Moreover, the map $\iota$ conjugates the actions of $\Sp(\Omega_{\pi'}|_{V'}, \Z)$ and $\Sp(\Omega_{\pi}|_{V}, \Z)$, and also the actions of $\RV(\pi')|_{V'}$ and a subgroup of $\RV(\pi)|_{V}$.
	\end{lem}
	\begin{proof}
		It is obvious from the definitions that $\iota$ is a symplectic isomorphism between $V'$ and $V$, so $\iota$ conjugates the actions of $\Sp(\Omega_{\pi'}|_{V'}, \Z)$ and $\Sp(\Omega_{\pi}|_{V}, \Z)$. Since the genus is preserved, one has that $V$ is a complement of $\ker \Omega_\pi$. 
		
		Now, let $\gamma'$ be an arrow in the Rauzy class of $\pi'$ starting at $\pi'$ and let $\gamma = \mathcal{E}_*(\gamma')$. Assume that $\gamma$ is constructed as case (1) in the definition. Let $\beta$ and $\beta'$ be the last letters on the top and bottom rows of $\pi'$, respectively. Then, by definition, $B_{\gamma'}^{-1} = \Id - E_{\beta' \beta}$ and $B_\gamma^{-1} = (\Id - E_{\alpha \beta})(\Id - E_{\beta' \beta}) = \Id - E_{\beta' \beta} - E_{\alpha \beta}$. From these relations, it is easy to see that $\iota(u) B_\gamma^{-1} = \iota(u)(\Id - E_{\beta' \beta})$ for every $u \in V'$, so $\iota(u B_{\gamma'}^{-1}) = \iota(u) B_{\gamma}^{-1}$ for every $u \in V'$. Similar computations for cases (2) and (3) show that $\iota_*$ is a monomorphism mapping $\RV(\pi')|_{V'}$ to a subgroup of $\RV(\pi)|_{V}$.
	\end{proof}
	
	In particular, we obtain that $[\Sp(\Omega_{\pi})|_{V} : \RV(\pi)|_{V}] \leq [\Sp(\Omega_{\pi'})|_{V'} : \RV(\pi')|_{V'}]$.
	
	We also have that genus-preserving simple extensions preserve the spin parity:
	\begin{lem}
		Let $\pi$ be a genus-preserving simple extension of $\pi'$. Then, the Arf invariants of $Q_{\pi}$ and $Q_{\pi'}$ coincide.
	\end{lem}
	
	\begin{proof}
		Let $\iota \colon H_1(M_{\pi'} \setminus \Sigma_{\pi'}) \to H_1(M_{\pi} \setminus \Sigma_{\pi})$ be the natural injection. Then, one has that $Q_{\pi'}(v) = Q_{\pi}(\iota(v))$ and $\langle v, w\rangle_{\pi'} = \langle \iota(v), \iota(w)\rangle_{\pi}$ for every $v,w \in H_1(M_{\pi'} \setminus \Sigma_{\pi'})$. Since the genus is preserved, we obtain that if $(v_\alpha, w_\alpha)_{\alpha \in \mathcal{B}}$ is a maximal symplectic subset of $H_1(M_{\pi'} \setminus \Sigma_{\pi'})$, then $(\iota(v_\alpha),\iota (w_\alpha))_{\alpha \in \mathcal{B}}$ is a maximal symplectic subset of $H_1(M_{\pi} \setminus \Sigma_{\pi})$ and we conclude that
		\[
			\mathrm{Arf}(Q_{\pi'}) = \sum_{\alpha \in \mathcal{B}} Q_\pi(v_\alpha)Q_\pi(w_\alpha) = \sum_{\alpha \in \mathcal{B}} Q_{\pi}(\iota(v_\alpha))Q_{\pi}(\iota(w_\alpha)) = \mathrm{Arf}(Q_{\pi}).
		\]
	\end{proof}
	
	We say that a permutation $\pi$ is \emph{standard} if $\pi_{\mathrm{b}}(\alpha_{\mathrm{t}, d}) = 1$ and $\pi_{\mathrm{b}}(\alpha_{\mathrm{t}, 1}) = d$. Observe that such a permutation is always irreducible. It was proven by Rauzy that standard permutations exist in every Rauzy class \cite{R:echanges}. The following lemma asserts the existence of some genus-preserving simple extensions for such permutations:
	
	\begin{lem} \label{prop:reduction}
		Let $\pi = (\pi_\mathrm{t}, \pi_\mathrm{b})$ be a standard permutation from $\A$ to $\{1, \dotsc, d\}$. Assume that $M_{\pi} \in \H(m_1, \dotsc, m_n)$ where $m_1 \geq 2$ and $m_i \geq 1$ for every $2 \leq i \leq n$. Then, there exists an irreducible permutation $\pi'$ such that $M_{\pi'} \in \H(m_{1,1}, m_{1,2}, m_2, \dotsc, m_n)$ and such that $\pi'$ is a simple extension of $\pi$, where $m_{1,1}, m_{1,2} \geq 1$ are any integers satisfying $m_{1,1} + m_{1,2} = m_1$.
	\end{lem}
	
	\begin{proof}		
		Let $p \colon P_{\pi} \to M_{\pi}$ be the projection map obtained by identifying the sides of the polygon $P_{\pi}$ by translation.
	
		Consider the bijection $s \colon \A \times \{ \mathrm{t}, \mathrm{b} \} \to \A \times \{ \mathrm{t}, \mathrm{b} \}$ defined by:
		\begin{itemize}
			\item $s(\alpha_{\mathrm{t}, j}, \mathrm{t}) = (\alpha_{\mathrm{t}, j-1}, \mathrm{b})$ if $j > 1$;
			\item $s(\alpha_{\mathrm{t}, 1}, \mathrm{t}) = (\alpha_{\mathrm{t}, d}, \mathrm{t})$;
			\item $s(\alpha_{\mathrm{b}, j}, \mathrm{b}) = (\alpha_{\mathrm{b}, j+1}, \mathrm{t})$ if $j < d$; and
			\item $s(\alpha_{\mathrm{b}, d}, \mathrm{b}) = (\alpha_{\mathrm{b}, 1}, \mathrm{b})$.
		\end{itemize}
		
		The bijection $s$ encodes the process of turning around a marked point in a clockwise manner. Indeed, the set $\A \times \{ \mathrm{t} \}$ corresponds to the top sides of $P_\pi$, while the set $\A \times \{ \mathrm{b} \}$ corresponds to the bottom sides. For $\alpha \in \A$, let $z \in P_\pi$ be its left endpoint. The orbit of $(\alpha, \mathrm{t})$ by $s$ is equal to the set of top sides of $P_\pi$ whose left endpoints $z'$ satisfy $p(z') = p(z)$ and the bottom sides of $P_\pi$ whose right endpoints $z'$ satisfy $p(z') = p(z)$.
		
		We can use the orbit by $s$ to compute the conical angle of $p(z)$. Indeed, it is easy to see that such angle is equal to $\pi |\mathrm{Orb}_s( \alpha, \mathrm{t} ) \setminus \{ (\alpha_{\mathrm{t}, 1}, \mathrm{t}), (\alpha_{\mathrm{b}, d}, \mathrm{b}) \}|$.
		
		 Now, let $z \in \C$ be a top vertex of $P_{\pi}$ such that $p(z)$ is the conical singularity of order $m_1$. By using that $\pi$ is standard, we can assume that $z \neq 0, d$. Indeed, if $p(0) = p(d)$, then the top $\alpha_{\mathrm{t}, 1}$-side of $P_{\pi}$ joins $0$ and a vertex $z \in P_{\pi} \setminus \{0, d\}$ such that $p(z) = p(0) = p(d)$. Let $\alpha \in \A$ such that $z$ is the left endpoint of the top $\alpha$-side of $P_\pi$, which exists since $z \neq d$. We have that $\alpha \neq \alpha_{1, \mathrm{t}}$ since $z \neq 0$.
		
		Consider the set $\mathrm{Orb}_s( \alpha, \mathrm{t} ) \setminus \{ (\alpha_{\mathrm{t}, 1}, \mathrm{t}), (\alpha_{\mathrm{b}, d}, \mathrm{b}) \}$, ordered by the order its elements occur when applying  $s$ to $(\alpha, \mathrm{t})$ iteratively. Observe that an element is at an odd position if and only if it corresponds to a top side. Moreover, observe that its cardinality is equal to $2 + 2m_1 = 2 + 2m_{1,1} + 2m_{1,2}$. Let $(\beta, \mathrm{t})$ be the $(3 + 2m_{1,1})$-th element. We define the simple extension $\pi'$ of $\pi$ by inserting a letter $\alpha' \notin \A$ before $\alpha$ in the top row and before $\beta$ in the bottom row. Let $\A' = \A \cup \{\alpha'\}$.
		
		We will now prove that $M_{\pi'} \in \H(m_{1,1}, m_{1,2}, m_2, \dotsc, m_n)$. Indeed, consider the bijection $s' \colon \A' \times \{ \mathrm{t}, \mathrm{b} \} \to \A' \times \{ \mathrm{t}, \mathrm{b} \}$ defined for $\pi'$ in an analogous way as $s$ for $\pi$. It is easy to see that:
		\begin{itemize}
			\item $s'(\alpha', \mathrm{t}) = s(\alpha, \mathrm{t})$;
			\item $s'(\alpha, \mathrm{t}) = (\alpha', \mathrm{b})$;
			\item $s'(\alpha', \mathrm{b}) = (\beta, \mathrm{t})$;
			\item if $\beta = \alpha_{\mathrm{b},1}$, then $s'(\alpha_{\mathrm{t}, 1}, \mathrm{t}) = (\alpha', \mathrm{t})$. Otherwise, let $\beta' \in \A$ be the letter before $\beta$ in $\pi_{\mathrm{b}}$ and then $s'(\beta', \mathrm{b}) = (\alpha', \mathrm{t})$;
		\end{itemize}
		and that $s'$ coincides with $s$ elsewhere.
		
		Let $k$ be the smallest natural number satisfying $s^{k+1}(\alpha, \mathrm{t}) = (\beta, \mathrm{t})$. If $\beta = \alpha_{\mathrm{b},1}$, then $s^k(\alpha, \mathrm{t}) = (\alpha_{\mathrm{t}, 1}, \mathrm{t})$ and, otherwise, $s^k(\alpha, \mathrm{t}) = (\beta', \mathrm{b})$. In both cases, $s'(s^k(\alpha, \mathrm{t})) = (\alpha', \mathrm{t})$. Moreover, since we have that $s'(\alpha', \mathrm{t}) = s(\alpha, \mathrm{t})$ we obtain that the orbit of $(\alpha', \mathrm{t})$ by $s'$ is:
		\[
			(\alpha', \mathrm{t}), s(\alpha, \mathrm{t}), s^2(\alpha, \mathrm{t}), \dotsc, s^k(\alpha, \mathrm{t}),
		\]
		so, by the choice of $\beta$, $|\mathrm{Orb}_{s'}( \alpha', \mathrm{t} ) \setminus \{ (\alpha_{\mathrm{t}, 1}, \mathrm{t}), (\alpha_{\mathrm{b}, d}, \mathrm{b}) \}| = 2 + 2m_{1,1}$.

		On the other hand, let $\ell$ be the smallest natural number satisfying $s^{\ell + 1}(\beta, \mathrm{t}) = (\alpha, \mathrm{t})$. The orbit of $(\alpha, \mathrm{t})$ by $s'$ is:
		\[
			(\alpha, \mathrm{t}), (\alpha', \mathrm{b}), (\beta, \mathrm{t}), s(\beta, \mathrm{t}), s^2(\beta, \mathrm{t}), \dotsc, s^\ell(\beta, \mathrm{t}),
		\]
		so $|\mathrm{Orb}_{s'}( \alpha, \mathrm{t} ) \setminus \{ (\alpha_{\mathrm{t}, 1}, \mathrm{t}), (\alpha_{\mathrm{b}, d}, \mathrm{b}) \}| = 2 + 2m_1 - (2 + 2m_{1,1}) + 2 = 2 + 2m_{1,2}$.
		
		These two orbits are disjoint, so the $\alpha'$-side of $M_{\pi'}$ joins two distinct conical singularities of orders $m_{1,1}$ and $m_{1,2}$. Since $s'$ coincides with $s$ outside of these orbits, the orders of the rest of the conical singularities are preserved.
	\end{proof}
	
	We obtain the following corollary:
	\begin{cor}
		Let $\pi$ be a standard permutation such that $M_{\pi}$ has genus $g$. Let $V$ be a complement of $\ker \Omega_\pi$. We have that:
		
		\begin{itemize}
			\item If $g \geq 3$ and $M_\pi$ belongs to a connected stratum, then $\RV(\pi)|_{V} = \Sp(\Omega_\pi|_{V}, \Z)$.
			\item If $M_\pi$ belongs to $\H(2m_1, \dotsc, 2m_n)^{\mathrm{spin}}$, where $\mathrm{spin} \in \{ \mathrm{even}, \mathrm{odd} \}$, $n \geq 2$ and $m_i \geq 1$ for every $1 \leq i \leq n$, then $\RV(\pi)|_{V}$ contains an isomorphic copy of the Rauzy--Veech group of $\H(2g - 2)^{\mathrm{spin}}$ and has index at most $2^{g-1}(2^g \pm 1)$ inside its ambient symplectic group.
		\end{itemize}
		
	\end{cor}
	
	\begin{proof}
		Assume first that $M_{\pi}$ belongs to a connected stratum. Let $\pi'$ be a standard permutation representing some connected component of $\H(2g - 2)$. By iterating the previous lemma and the fact that every Rauzy class contains standard permutations, there exists a sequence of simple extensions and Rauzy inductions taking $\pi'$ to $\pi$. Therefore, we obtain that $\RV(\pi)|_{V}$ contains an isomorphic copy of $\RV(\pi')$. Recall that the Rauzy--Veech groups of non-hyperelliptic components are maximal subgroups of $\Sp(2g, \Z)$ \cite[Theorem 3]{BGP:finiteindex}. If $g = 3$, the index of the Rauzy--Veech group of $\H(4)^{\odd}$ is 28, while the index of the Rauzy--Veech group of $\H(4)^{\mathrm{hyp}}$ is 288, which is not divisible by 28. Therefore, $\RV(\pi)|_{V} = \Sp(\Omega_\pi|_{V}, \Z)$ in this case by maximality. If $g \geq 4$, $\RV(\pi)|_{V}$ contains isomorphic copies of the Rauzy--Veech groups of both non-hyperelliptic components of $\H(2g - 2)$, so we also conclude by maximality.
		
		Now, if $M_{\pi}$ belongs to $\H(2m_1, \dotsc, 2m_n)^{\mathrm{spin}}$ let $\pi'$ be a standard permutation representing $\H(2g - 2)^{\mathrm{spin}}$. By iterating the previous lemma and using the fact that the Arf invariants of the quadratic forms are preserved by simple extensions, there exists a sequence of simple extensions and Rauzy inductions taking $\pi'$ to $\pi$, which completes the proof.
	\end{proof}
	
	The following lemma completes our classification of the Rauzy--Veech groups: 
	
	\begin{lem}
		The Rauzy--Veech group $\H(2m_1, \dotsc, 2m_n)^{\mathrm{spin}}$ restricted to a complement of the kernel of the symplectic form, where $\mathrm{spin} \in \{ \mathrm{even}, \mathrm{odd} \}$, $n \geq 2$ and $m_i \geq 1$ for every $1 \leq i \leq n$, is isomorphic to the Rauzy--Veech group of $\H(2g - 2)^{\mathrm{spin}}$ and has index $2^{g-1}(2^g \pm 1)$ inside its ambient symplectic group.
	\end{lem}
	
	\begin{proof}
		We will use the explicit permutation representatives of such strata computed by Zorich \cite[Proposition 3, Proposition 3]{Z:representatives}. We start by proving the result for $\H(2, 2, \dotsc, 2)^{\mathrm{odd}}$. Let $g \geq 3$ and consider the representative:
		{\small\[
			\tau =
			\begin{pmatrix}
				0 & 1 & 2 & 3 & 4 & 5 & 6 & 7 & \cdots & 3g - 7 & 3g - 6 & 3g - 5 & 3g - 4 & 3g - 3 \\
				3 & 2 & 4 & 6 & 5 & 7 & 9 & 8 & \cdots & 3g - 5 & 3g - 3 & 3g - 4 & 1 & 0
			\end{pmatrix}.
		\]}
		It is easy to see that $\{ e_0 - e_1 + e_{3k+1}  \}_{k = 1}^{g-2}$ is a basis of $\ker \Omega_\tau$ consisting of singular vectors for $Q_\tau$. Therefore, $\ker \Omega_\tau \subseteq \mathrm{S}(Q_\tau)$. Let $V \subseteq H_1(M_\tau \setminus \Sigma_\tau)$ be the $\Z$-module spanned by $\{ e_0, e_1, e_2, e_3, e_5, e_6, e_8, e_9 \dotsc, e_{3g - 4}, e_{3g - 3} \}$ which is a complement of $\ker \Omega_\tau$. For any $S \in \bRV(\tau)$ we have that, for any $u \in \bV$,
		\begin{align*}
			Q_\tau(S^1(u)) &= Q_\tau(S^1(u) + S^0(u) + S^0(u)) = Q_\tau(S(u) + S^0(u)) \\
			&= Q_\tau(S(u)) + Q_\tau(S^0(u)) + \langle S(u), S^0(u) \rangle = Q_\tau(u),
		\end{align*}
		where $S^0$ is the projection of $S|_{\text{\bV}}$ on $\ker \bOmega_\tau$. Therefore, $\RV(\tau)|_V \subseteq O(Q|_V)$. By the previous corollary, we obtain that $\RV(\tau)|_V$ is isomorphic to and has the same index as $\RV(\tau^{(2g)})$. Now, a representative of any connected component of the form $\H(2m_1, \dotsc, 2m_n)^{\mathrm{odd}}$ can be obtained by genus-preserving simple reductions of $\tau$ which consist of erasing some letters of the form $3k+1$ for $1 \leq k \leq g - 2$. These simple reductions are also simple extensions when reversed, so the index of the Rauzy--Veech group of $\H(2, 2, \dotsc, 2)^{\mathrm{odd}}$ cannot be larger than the index of the Rauzy--Veech group of $\H(2m_1, \dotsc, 2m_n)^{\mathrm{odd}}$. We conclude by the previous corollary.
		
		The proof for the even components is completely analogous by using the representative:
		{\small\[
			\sigma =
			\begin{pmatrix}
				0 & 1 & 2 & 3 & 4 & 5 & 6 & 7 & \cdots & 3g - 7 & 3g - 6 & 3g - 5 & 3g - 4 & 3g - 3 \\
				6 & 5 & 4 & 3 & 2 & 7 & 9 & 8 & \cdots & 3g - 5 & 3g - 3 & 3g - 4 & 1 & 0
			\end{pmatrix},
		\]}
		the basis $\{ \sum_{\alpha = 0}^6 (-1)^\alpha e_\alpha \} \cup \{ e_1 - e_2 + e_{3k+1}  \}_{k = 2}^{g-2}$ of $\ker \Omega_\sigma$ consisting of singular vectors for $Q_\sigma$ and the complement $V$ of $\ker \Omega_\sigma$ defined as in the previous case.
	\end{proof}
	
	Finally, we will recall some concepts to state the classification theorem using a more common notation. We refer the reader to the survey by Viana \cite{V:iet} and the lecture notes by Yoccoz \cite{Y:pisa} for more details.
	
	Let $\pi$ by an irreducible, nondegenerate permutation. There exist natural maps
	\[
		H_1(M_\pi \setminus \Sigma_\pi) \to H_1(M_\pi) \to H_1(M_\pi, \Sigma_\pi),
	\]
	the former being surjective and the latter injective. The basis $\{e_\alpha\}_{\alpha \in \A} = ([\theta_\alpha])_{\alpha \in \A}$ of $H_1(M_\pi \setminus \Sigma_\pi)$ that we have thus far considered is dual to a basis $(f_\alpha)_{\alpha \in \A} = ([\zeta_\alpha])_{\alpha \in \A}$ of $H_1(M_\pi, \Sigma_\pi)$ consisting on the sides of the polygon $P_\pi$. The image of $H_1(M_\pi)$ inside $H_1(M_\pi, \Sigma_\pi)$ is the set $H(\pi)$ spanned by $(\Omega_\pi f_\alpha )_{\alpha \in \A}$ (acting on column vectors). There is a natural nondegenerate symplectic form defined on $H(\pi)$ by $\langle \Omega_\pi v, \Omega_\pi w \rangle = v^\tr \Omega_\pi w$.
	
	The duality between the bases allows us to define an action of $\RV(\pi)$ on $H_1(M_\pi, \Sigma_\pi)$, which can be understood as the action of $\RV(\pi)$ on \emph{column} vectors. This action preserves $H(\pi)$, so we can consider its restriction $\RV(\pi)|_{H(\pi)}$ to $H(\pi)$, which is a subgroup of the symplectic group $\Sp(\Omega_\pi|_{H(\pi)}, \Z)$ induced by the symplectic form on $H(\pi)$.
	
	We have classified the Rauzy--Veech groups acting on $H_1(M_\pi \setminus \Sigma_\pi)$ or, equivalently, its action on \emph{row} vectors. Recalling \Cref{thm:minimal}, \Cref{thm:odd_d} and \Cref{lem:transvections_rauzy_class} and combining our results with those of Avila, Matheus and Yoccoz \cite{AMY:hyperelliptic} we can summarize the full classification of the group $\RV(\pi)|_{H(\pi)}$:
	\smallbreak
	\begin{thm} \label{thm:classification}
		For a connected component $\mathcal{C}$ of a stratum of genus $g$ translation surfaces, we denote by $\RV(\mathcal{C})$ its associated Rauzy--Veech group and by $\RV(\mathcal{C})|_H$ the restriction of $\RV(\mathcal{C})$ to the image of the symplectic form. We write $\Sp(2g, \Z)$ for its ambient symplectic group. We have the following:
		\begin{itemize}
			\item If $\mathcal{C}$ is hyperelliptic, then the group $\RV(\mathcal{C})$ is characterized by preserving modulus two, as an action of column vectors, a specific finite set \cite[Theorem 2.9]{AMY:hyperelliptic}.
			\item The group $\RV(\H(2m_1, \dotsc, 2m_n)^{\mathrm{spin}})|_H$, where $\mathrm{spin} \in \{ \mathrm{even}, \mathrm{odd} \}$, $n \geq 2$ and $m_i \geq 1$ for every $1 \leq i \leq n$, is the subgroup of $\Sp(2g, \Z)$ whose modulus-two reduction preserves, as an action on row vectors, a specific quadratic form whose Arf invariant is $\mathrm{spin}$. It has index $2^{g-1}(2^g \pm 1)$ inside $\Sp(2g, \Z)$.
			\item For any other connected component $\mathcal{C}$, including every connected stratum for $g \geq 3$, one has that $\RV(\mathcal{C})|_H = \Sp(2g, \Z)$.
		\end{itemize}
		Moreover, $\RV(\mathcal{C})|_H$ is generated by symplectic transvections along canonical vectors.
	\end{thm}
	
	\section*{Appendix} \label{sec:appendix}
	
	In this section we will explicitly state some computations for the base cases of the induction used to prove \Cref{thm:minimal}. Specifically, these facts were used to show that some permutations belong to the desired minimal strata and that the modulus-two reduction of the Rauzy--Veech groups are the entirety of the orthogonal groups.
	
	It is possible to find all elements of $\NS(\tau^{(d)})$ for small values of $d$ by hand. This was done to find the base cases of the induction in \Cref{lem:representatives}. In particular:
	\begin{fact}
		$\NS(\tau^{(6)})$ consists of 36 elements, which, written on the basis $(\be_\alpha)_{\alpha = 1}^6$, are:
		{\scriptsize
		\begin{align*}
			&(1,0,0,0,0,0),\quad(0,1,0,0,0,0),\quad(1,1,0,0,0,0),\quad(0,0,1,0,0,0),\quad(1,0,1,0,0,0),\quad(0,1,1,0,0,0) \\
			&(0,0,0,1,0,0),\quad(1,0,0,1,0,0),\quad(1,1,0,1,0,0),\quad(1,0,1,1,0,0),\quad(0,0,0,0,1,0),\quad(1,0,0,0,1,0) \\
			&(1,1,0,0,1,0),\quad(1,0,1,0,1,0),\quad(0,0,0,1,1,0),\quad(1,1,1,1,1,0),\quad(0,0,0,0,0,1),\quad(1,0,0,0,0,1) \\
			&(0,1,0,0,0,1),\quad(0,0,1,0,0,1),\quad(0,0,0,1,0,1),\quad(0,1,0,1,0,1),\quad(1,1,0,1,0,1),\quad(0,0,1,1,0,1) \\
			&(1,0,1,1,0,1),\quad(1,1,1,1,0,1),\quad(0,0,0,0,1,1),\quad(0,1,0,0,1,1),\quad(1,1,0,0,1,1),\quad(0,0,1,0,1,1) \\
			&(1,0,1,0,1,1),\quad(1,1,1,0,1,1),\quad(1,1,0,1,1,1),\quad(1,0,1,1,1,1),\quad(0,1,1,1,1,1),\quad(1,1,1,1,1,1).
		\end{align*}
		 }
	\end{fact}

	\begin{fact}
		$\NS(\sigma^{(8)})$ consists of 120 elements, which, written on the basis $(\be_\alpha)_{\alpha = 1}^8$, are:
		{\scriptsize
		\begin{align*}
			&(1,0,0,0,0,0,0,0),\quad(0,1,0,0,0,0,0,0),\quad(1,1,0,0,0,0,0,0),\quad(0,0,1,0,0,0,0,0) \\
			&(1,0,1,0,0,0,0,0),\quad(0,1,1,0,0,0,0,0),\quad(0,0,0,1,0,0,0,0),\quad(1,0,0,1,0,0,0,0) \\
			&(0,1,0,1,0,0,0,0),\quad(0,0,1,1,0,0,0,0),\quad(0,0,0,0,1,0,0,0),\quad(1,0,0,0,1,0,0,0) \\
			&(0,1,0,0,1,0,0,0),\quad(0,0,1,0,1,0,0,0),\quad(0,0,0,1,1,0,0,0),\quad(1,1,1,1,1,0,0,0) \\
			&(0,0,0,0,0,1,0,0),\quad(1,0,0,0,0,1,0,0),\quad(1,1,0,0,0,1,0,0),\quad(1,0,1,0,0,1,0,0) \\
			&(1,0,0,1,0,1,0,0),\quad(0,1,1,1,0,1,0,0),\quad(1,0,0,0,1,1,0,0),\quad(0,1,1,0,1,1,0,0) \\
			&(0,1,0,1,1,1,0,0),\quad(0,0,1,1,1,1,0,0),\quad(0,1,1,1,1,1,0,0),\quad(1,1,1,1,1,1,0,0) \\
			&(0,0,0,0,0,0,1,0),\quad(1,0,0,0,0,0,1,0),\quad(1,1,0,0,0,0,1,0),\quad(1,0,1,0,0,0,1,0) \\
			&(1,0,0,1,0,0,1,0),\quad(0,1,1,1,0,0,1,0),\quad(1,0,0,0,1,0,1,0),\quad(0,1,1,0,1,0,1,0) \\
			&(0,1,0,1,1,0,1,0),\quad(0,0,1,1,1,0,1,0),\quad(0,1,1,1,1,0,1,0),\quad(1,1,1,1,1,0,1,0) \\
			&(0,0,0,0,0,1,1,0),\quad(1,1,1,0,0,1,1,0),\quad(1,1,0,1,0,1,1,0),\quad(1,0,1,1,0,1,1,0) \\
			&(0,1,1,1,0,1,1,0),\quad(1,1,1,1,0,1,1,0),\quad(1,1,0,0,1,1,1,0),\quad(1,0,1,0,1,1,1,0) \\
			&(0,1,1,0,1,1,1,0),\quad(1,1,1,0,1,1,1,0),\quad(1,0,0,1,1,1,1,0),\quad(0,1,0,1,1,1,1,0) \\
			&(1,1,0,1,1,1,1,0),\quad(0,0,1,1,1,1,1,0),\quad(1,0,1,1,1,1,1,0),\quad(0,1,1,1,1,1,1,0) \\
			&(0,0,0,0,0,0,0,1),\quad(1,0,0,0,0,0,0,1),\quad(0,1,0,0,0,0,0,1),\quad(0,0,1,0,0,0,0,1) \\
			&(0,0,0,1,0,0,0,1),\quad(1,1,1,1,0,0,0,1),\quad(0,0,0,0,1,0,0,1),\quad(1,1,1,0,1,0,0,1) \\
			&(1,1,0,1,1,0,0,1),\quad(1,0,1,1,1,0,0,1),\quad(0,1,1,1,1,0,0,1),\quad(1,1,1,1,1,0,0,1) \\
			&(0,0,0,0,0,1,0,1),\quad(0,1,0,0,0,1,0,1),\quad(1,1,0,0,0,1,0,1),\quad(0,0,1,0,0,1,0,1) \\
			&(1,0,1,0,0,1,0,1),\quad(1,1,1,0,0,1,0,1),\quad(0,0,0,1,0,1,0,1),\quad(1,0,0,1,0,1,0,1) \\
			&(1,1,0,1,0,1,0,1),\quad(1,0,1,1,0,1,0,1),\quad(0,0,0,0,1,1,0,1),\quad(1,0,0,0,1,1,0,1) \\
			&(1,1,0,0,1,1,0,1),\quad(1,0,1,0,1,1,0,1),\quad(1,0,0,1,1,1,0,1),\quad(0,1,1,1,1,1,0,1) \\
			&(0,0,0,0,0,0,1,1),\quad(0,1,0,0,0,0,1,1),\quad(1,1,0,0,0,0,1,1),\quad(0,0,1,0,0,0,1,1) \\
			&(1,0,1,0,0,0,1,1),\quad(1,1,1,0,0,0,1,1),\quad(0,0,0,1,0,0,1,1),\quad(1,0,0,1,0,0,1,1) \\
			&(1,1,0,1,0,0,1,1),\quad(1,0,1,1,0,0,1,1),\quad(0,0,0,0,1,0,1,1),\quad(1,0,0,0,1,0,1,1) \\
			&(1,1,0,0,1,0,1,1),\quad(1,0,1,0,1,0,1,1),\quad(1,0,0,1,1,0,1,1),\quad(0,1,1,1,1,0,1,1) \\
			&(1,1,0,0,0,1,1,1),\quad(1,0,1,0,0,1,1,1),\quad(0,1,1,0,0,1,1,1),\quad(1,1,1,0,0,1,1,1) \\
			&(1,0,0,1,0,1,1,1),\quad(0,1,0,1,0,1,1,1),\quad(1,1,0,1,0,1,1,1),\quad(0,0,1,1,0,1,1,1) \\
			&(1,0,1,1,0,1,1,1),\quad(0,1,1,1,0,1,1,1),\quad(1,0,0,0,1,1,1,1),\quad(0,1,0,0,1,1,1,1) \\
			&(1,1,0,0,1,1,1,1),\quad(0,0,1,0,1,1,1,1),\quad(1,0,1,0,1,1,1,1),\quad(0,1,1,0,1,1,1,1) \\
			&(0,0,0,1,1,1,1,1),\quad(1,0,0,1,1,1,1,1),\quad(0,1,0,1,1,1,1,1),\quad(0,0,1,1,1,1,1,1).
		\end{align*}}
	\end{fact}
	Proving that $\{\be_\alpha\}_{\alpha \in \A}^{Q^{(d)}} = \NS(\pi^{(d)})$ for the base cases of the induction can be done as follows: we start with the set $S_1 = \{\be_\alpha\}_{\alpha \in \A}$, which is contained in $\{\be_\alpha\}_{\alpha \in \A}^{Q^{(d)}}$. We define the set $S_{k+1}$ as the union of $S_k$ with the set of vectors of the form $v + \be_\alpha$, where $v \in S_k$ and $Q(v + \be_\alpha) = 1$. Clearly, $S_k \subseteq \{\be_\alpha\}_{\alpha \in \A}^{Q^{(d)}}$ for each $k$ and, after a small number of iterations, it must be equal to $\{\be_\alpha\}_{\alpha \in \A}^{Q^{(d)}}$. We can then check that $\{\be_\alpha\}_{\alpha \in \A}^{Q^{(d)}} = \NS(\pi^{(d)})$.
	\bigbreak
	
	\textbf{Acknowledgements:} I am grateful to Giovanni Forni for his interesting questions and remarks which motivated the newer versions of the article. I'm also grateful to my advisors, Anton Zorich and Carlos Matheus.
	
\sloppy\printbibliography
	
\end{document}